\documentclass[oneside,a4paper]{amsart}
\usepackage[margin=3cm]{geometry}
\usepackage{graphicx,amssymb}
\usepackage{xcolor}
\usepackage{amsthm,amsfonts,amssymb,latexsym,epsfig,mathrsfs,yfonts,marvosym,latexsym,epsfig}
\usepackage{amsmath}

\usepackage{mathtools}
\mathtoolsset{showonlyrefs=true}
\usepackage{hyperref}
\hypersetup{bookmarksdepth=3,
            colorlinks = true,%
            citecolor = {blue},%
            linkcolor = {red}%
}%
\usepackage{circuitikz}
\usepackage{multicol}
\usepackage[all]{xy}
\AtEndDocument{\vfill\eject\batchmode}
\usepackage{euflag}
\DeclareMathAlphabet\oldmathcal{OMS}        {cmsy}{b}{n}
\SetMathAlphabet\oldmathcal{normal}{OMS}{cmsy}{m}{n}
\DeclareMathAlphabet\oldmathbcal{OMS}       {cmsy}{b}{n}
\usepackage[active]{srcltx}
\usepackage{eucal}

\newtheorem{theorem}{Theorem}[section]
\newtheorem{lemma}[theorem]{Lemma}
\newtheorem{proposition}[theorem]{Proposition}
\newtheorem{corollary}[theorem]{Corollary}

\newtheorem{question}[theorem]{Question}
\newtheorem{def/prop}[theorem]{Definition/Proposition}

\theoremstyle{definition}
\newtheorem{definition}[theorem]{Definition}
\newtheorem{remark}[theorem]{Remark}

\newtheorem{example}{Example}[section]

\DeclareSymbolFont{bbold}{U}{bbold}{m}{n}
\DeclareSymbolFontAlphabet{\mathbbold}{bbold}

\def\BOne{\mathchoice{\scalebox{1.16}{$\displaystyle\mathbbold 1$}}{\scalebox{1.16}{$\textstyle\mathbbold 1$}}{\scalebox{1.16}{$\scriptstyle\mathbbold 1$}}{\scalebox{1.16}{$\scriptscriptstyle\mathbbold 1$}}}
\def\fract#1#2{\raise4pt\hbox{$ #1 \atop #2 $}}

\def\bbc{{\mathbb C}}

\def\bbn{{\mathbb N}}

\def\bbp{{\mathbb P}}
\def\bbq{{\mathbb Q}}
\def\bbr{{\mathbb R}}

\def\bbt{{\mathbb T}}

\def\bbz{{\mathbb Z}}
\def\phi{\varphi}

\def\bfV{{\bf V}}

\def\bfS{{\bf S}}

\def\cald{{\mathcal D}}

\def\calf{{\mathcal F}}

\def\call{{\mathcal L}}

\def\calo{{\mathcal O}}

\def\caly{{\mathcal Y}}

\def\calS{{\mathcal S}}

\def\calC{{\oldmathcal C}}

\def\calL{{\oldmathcal L}}

\def\gg{{\mathfrak g}}

\def\gz{{\mathfrak z}}

\def\v{{\rm{v}}}

\def\<{\langle}
\def\>{\rangle}
\def\ra#1{\to}

\newcommand{\Aut}{\mathrm{Aut}}

\newcommand{\transpose}{\intercal}
\newcommand{\Tr}{\operatorname{Tr}}
\newcommand{\Hess}{\operatorname{Hess}}
\DeclarePairedDelimiter{\norm}{\lVert}{\rVert}
\DeclarePairedDelimiter{\abs}{\lvert}{\rvert}

\def\fract#1#2{\raise4pt\hbox{$ #1 \atop #2 $}}

\def\hook{\mathbin{\hbox to 6pt{%
                 \vrule height0.4pt width5pt depth0pt
                 \kern-.4pt
                 \vrule height6pt width0.4pt depth0pt\hss}}}

\def\ra{\rightarrow}
\def\kt{\mathfrak{t}}

\def\End{\rm{End}}

\def\scal{{\mathrm{scal}}}
\def\Scaltot{\mathbf{Scal}}
\def\Voltot{\mathbf{Vol}}
\newcommand{\CRYam}{\mathcal{Y}_{CR}}
\newcommand{\EH}{\mathrm{EH}}
\newcommand{\EHmin}{\mathrm{EH}_{\mathrm{min}}}
\newcommand{\InfFun}{\mathbf{I}}
\DeclareMathOperator{\ScalTW}{\mathrm{Scal}^{\mathrm{TW}}}
\newcommand{\Reebcone}{\mathfrak{t}_+}
\newcommand{\Torus}{\mathbb{T}}
\newcommand{\tstX}{\mathcal{X}}
\newcommand{\tstL}{\mathcal{L}}%

\makeatletter
\@namedef{subjclassname@2020}{\textup{2020} Mathematics Subject Classification}
\makeatother

\title{The toric CR Yamabe problem}

\date{\today}

\author[E.\ Legendre]{Eveline Legendre}
\address[Eveline Legendre]{
Institut Camille Jordan\\
Universit\'e Claude Bernard Lyon 1\\
43 Boulevard du 11 novembre 1918\\
69622 Villeurbanne Cedex, France
}
\email{legendre@math.univ-lyon1.fr}

\author[C.\ Scarpa]{Carlo Scarpa}
\address[Carlo Scarpa]{
Institut Camille Jordan\\
Universit\'e Claude Bernard Lyon 1\\
43 Boulevard du 11 novembre 1918\\
69622 Villeurbanne Cedex, France
}
\email{scarpa@math.univ-lyon1.fr}

\author[C.\ T{\o}nnesen-Friedman]{Christina T{\o}nnesen-Friedman}
\address[Christina T{\o}nnesen-Friedman]{
Department of Mathematics\\
Union College\\
807 Union Street Schenectady\\
New York 12308
}
\email{tonnesec@union.edu}

\keywords{CR Yamabe problem, toric Sasaki manifolds}
\thanks{
    E.L.\ is supported by the ANR-FAPESP grant PRCI ANR-21-CE40-0017 and partially supported by a AAP grant of the University of Lyon~1.
    C.S.\ is supported by a MSCA Postdoctoral Fellowship, funded by the Research and Innovation framework programme Horizon Europe {\footnotesize\euflag} under grant agreement n.101149320.
}
\subjclass[2020]{53C25 (primary), 58E11, 53C18, 32Q15, 53D10, 32V20}

\begin{document}

\begin{abstract}
   We study the equivariant CR Yamabe problem for a fixed-point-free torus action on a co-oriented compact contact manifold. Such examples arise naturally in the study of Sasakian manifolds with constant scalar curvature. In the toric case, we show that this problem is equivalent to a kind of boundary value problem for an elliptic PDE on a pair of functions defined on a convex, polyhedral domain. We provide examples of solutions and prove that many contact toric manifolds admit compatible CR structures with distinct signs of CR Yamabe invariants. 
\end{abstract}

\maketitle

\setcounter{tocdepth}{1}
\tableofcontents

\section{Introduction}

In this paper, we study an equivariant version of the CR Yamabe problem on a compact pseudohermitian manifold. More precisely, we consider a compact co-oriented contact manifolds $(N,D)$ of real dimension $2n+1$, with an integrable CR structure $J\in\Gamma(\End\, D)$. The \emph{CR Yamabe problem}, introduced by Jerison and Lee \cite{JerisonLee_CRYamabe} consists in showing that there exists a contact form $\eta$ for $D$ which is \emph{positive} (that is compatible with the co-orientation of $(N,D)$, so that $(D,J,\eta)$ is a \emph{pseudohermitian} structure), and such that the \emph{Tanaka-Webster scalar curvature} of $(D,J,\eta)$ is constant
\begin{equation}
    \ScalTW(\eta,J) = \text{const.}
\end{equation}
One possible way of finding solutions of this problem is to minimise the so-called \emph{CR Einstein--Hilbert functional}
\begin{equation}
    \EH(\eta,J) \coloneqq \frac{\Scaltot(\eta,J)}{\Voltot(\eta,J)^{\frac{n}{n+1}}} \coloneqq \frac{\int_N \ScalTW(\eta)\,\eta\wedge(d\eta)^n}{\left(\int_N \eta\wedge(d\eta)^n\right)^{\frac{n}{n+1}}}
\end{equation}
over the space of positive contact forms of $(D,J)$, that is the conformal class of a fixed contact form. Indeed, one can show \cite{JerisonLee_CRYamabe} that for each positive contact form $\eta$ the following quantity, called the \emph{CR Yamabe energy}, is well defined
\begin{equation}\label{eq:CRYam_def}
    \CRYam(\eta,J)\coloneqq \inf_{f>0,\text{ smooth}} \EH(f^{-1}\eta,J),
\end{equation}
and that if any function $f$ realises the infimum in \eqref{eq:CRYam_def}, then $(D,J,f^{-1}\eta)$ is a pseudohermitian structure of constant Tanaka-Webster scalar curvature. This program was successfully carried out for $n>1$ to solve the CR Yamabe problem:~\cite{JerisonLee_CRYamabe,JerisonLee_normalcoords,JerisonLee_extremalsCRYam,Gamara_Yacoub}, while for $n=1$ it is still not clear in complete generality whether the infimum is achieved; nonetheless, the existence of pseudohermitian structures with constant Tanaka-Webster scalar curvature has been established~\cite{Gamara}.

\smallskip

As mentioned above, we are interested here in an \emph{equivariant} version of this problem, in which we consider a compact torus $\Torus$ of automorphisms of $(N,D,J_0,\eta_0)$, for some background $\Torus$-invariant pseudohermitian structure $(J_0,\eta_0)$, and look for $\Torus$-invariant pseudohermitian forms $\eta$ of constant Tanaka-Webster scalar curvature (cscTW, from now on). Other types of equivariant versions of the classical and CR Yamabe problems, with respect to a compact or a finite group, have been studied in the past: see \S\ref{sec:inv_yamabe} below for an overview of some results.  

There are many reasons for considering a $\Torus$-invariant version of the CR Yamabe problem: for example, we have a particular interest in Sasaki structures, a special case of pseudohermitian structures whose Reeb vector field is \emph{Killing}, i.e.\ its flow preserves the Riemannian metric (and thus the CR structure). These Sasaki structures are always invariant by a non-trivial torus (generated by their Reeb vector field). In what follows we will always assume that the background structure $(N,D,J_0,\eta_0)$ is Sasaki and that the Reeb vector field of $\eta_0$, denoted $\xi_0$, lies in $\kt \coloneqq \operatorname{Lie}\Torus$. Note in particular that this implies that the $\Torus$ action on $N$ has no fixed points.

With this setup, the existence of $\Torus$-invariant cscTW pseudo-hermitian structures on $(N,D,J_0)$ is ensured by a result of Zhang \cite{Zhang_Kcontact} and follows also from \cite{Ho_CRYam-equiv} as we explain in \S\ref{sec:inv_yamabe} below. Some questions remain open, such as the uniqueness or the isolation of solutions (see \cite{BHLTF20} for the Sasaki version of this problem), which is not known for the case of positive Tanaka-Webster scalar curvature.

\begin{remark}\label{remTwins}
    For a compact smooth manifold $N$ with a fixed CR structure $(N,D, J)$, the notion of \emph{Extremal Sasaki twins} was defined in \cite{BHLT25} to be two extremal Sasaki structures both compatible with $(D,J)$ and having commuting non-collinear Sasaki-Reeb vector fields. The special case where both of these extremal Sasaki structures have constant scalar curvature, considered in Section~$5.2$ (and Example $4$) of \cite{BHLT25}, might be considered as a (very) special case of non-uniqueness of $\bbt$-invariant cscTW pseudohermitian forms $\eta$. To the best of our knowledge, these are the only known examples of distinct $\Torus$-invariant cscTW structures in the same conformal class.
\end{remark}

On the one hand, the $\Torus$-invariance complicates the CR Yamabe problem: many of the classical results for the CR Yamabe problem rely on the study of so-called ``bubbles'' for the cscTW equation, which are \emph{not} invariant under the $\Torus$-action we consider. Developing new analytic and geometric tools to tackle this problem is an exciting challenge. On the other hand, in the \emph{toric} case that we consider here, i.e.\ when $\dim \Torus = (\dim N +1)/2$, \emph{the $\Torus$-invariant cscTW equation is equivalent to a certain non-linear PDE on functions defined on a compact, convex set of $\kt^*$}, see Proposition~\ref{prop:toricSasaki_intro} below.

We exploit this dimensional reduction to compute the toric CR Yamabe energy in some \emph{asymptotic} cases and provide many explicit examples of cscTW structures. We also give a positive answer to a question mentioned in \cite{SungTakeuchi_CRYam}: {\it given a fixed contact structure $D$ on a compact smooth manifold $N$, is it possible to find distinct compatible CR structures on $D$ with different signs of the CR Yamabe energy?}

\smallskip

In this paper, we will mostly be interested in the CR Yamabe energy as a function of the CR structure. More precisely, we fix a co-dimension $1$, co-oriented  contact distribution $D$ on $N$ and consider the set $\mathcal{CR}(N,D)^\Torus$ of $\Torus$-invariant integrable and pseudo-convex CR structures on $D$. Then the space of $\Torus$-invariant pseudo-hermitian structures on $D$ is the product of $\mathcal{CR}(N,D)^\Torus$ with the space of positive $\Torus$-invariant contact forms on $D$.

It follows from \cite{AfeltraMalchiodi_EH} and \cite{LahdiliLegendreScarpa_EHDF} that the critical points of the Einstein-Hilbert action on the space of $\Torus$-invariant pseudo-hermitian structures on $D$ are exactly the $\Torus$-invariant Sasaki structures of constant (transversal) scalar curvature, i.e.\ those cscTW structures that are Sasaki. We will see below that this result is easily proved in the toric case, see Theorem~\ref{theocritEH=cscS}.  Note that in \cite{LahdiliLegendreScarpa_EHDF} the point of view taken is the ``complex'' one, where one considers instead the pseudo-hermitian structures compatible with a fixed transversal holomorphic structure but vary the contact distribution. The two points of view are equivalent, see \cite[Appendix A]{ApostolovCalderbankLegendre_weighted}.

Given $J\in \mathcal{CR}(N,D)^\Torus$ and any $\Torus$-invariant contact form $\eta_0$ on $N$, the $\Torus$-invariant version of the CR Yamabe energy is
\begin{equation}\label{eq:CRYam_T_def}
    \CRYam^\Torus(J) \coloneqq \inf_{\substack{f>0, \text{ smooth},\\ \Torus\text{-invariant}}} \EH(f^{-1}\eta_0,J).
\end{equation} 
Any function realizing this infimum will give us a $\Torus$-invariant cscTW form conformal to $\eta_0$.

Compared to the classical CR Yamabe invariant, which is bounded above by the Yamabe invariant of the standard CR-sphere, the $\Torus$-equivariant Yamabe invariant of a CR-manifold is bounded above by a finer invariant, that we denote $\EH_{\min}$ and which is the \emph{minimum of the Einstein-Hilbert functional on the Sasaki-Reeb cone}. Indeed, as highlighted in \cite{LahdiliLegendreScarpa_CRYam}  \begin{equation}\label{eq:YamleqEmin}
    \CRYam^\Torus(J) \leq \EH_{\min}        \qquad \qquad  \forall J\in \mathcal{CR}(N,D)^\Torus.
\end{equation}  Moreover, they proved (again, from the complex point of view) that the bound being reached by an element $J$ guaranties the K-semistability of the associated Sasaki structure, as well as the existence of a  constant transversal scalar curvature Sasaki form (cscS, from now on). There has been tremendous developments recently around the existence of constant scalar curvature K\"ahler or Sasaki metrics (and more generally \emph{weighted cscK metrics} \cite{Lahdili_weighted}) in connection with K-stability. The toric case is particularly well-understood and establishes an equivalence between the existence of cscK (or cscS) structures and the positivity of a certain functional on the space of convex, piecewise-linear functions. We refer the reader to \cite{Donaldson_stability_toric} and the notes \cite{Apostolov_toric} for an overview of this theory.
Our results in this direction are still partial, but they suggest that at least in the toric setting, K-stability can be used to better understand the behavior of the toric CR Yamabe energy and formulate some interesting conjectures.

\subsection{Summary of results}
Our results hinge on the fact that in the toric setting, the contact manifold $(N,D)$, with Reeb field $\xi_0 \in \kt$ is completely characterized by a convex compact polytope $P_{\xi_0}$ of an affine space via the theory of Delzant~\cite{Delzant_polytope}, Lerman~\cite{Lerman:contactToric} and others, that we recall in Section~\ref{sec:toric}. Using the Sasaki version of Guillemin's seminal work \cite{Guillemin_toric}, this allows us to translate the $\Torus$-equivariant CR Yamabe problem in terms of functions defined on $P_{\xi_0}$. We briefly explain the main point of this construction, and we refer the reader to \S\ref{sss:ReviewSasaki} for more details. The polytope $P_{\xi_0}$ can be identified with a set of linear inequalities,
\begin{equation}
    P_{\xi_0}=\left\lbrace x\in\mathbb{R}^n \mid \ell_j(x)\geq 0,\ j=1,\dots,d \right\rbrace.
\end{equation}
The $\Torus$-invariant Sasaki structures on $(N,D)$ with Reeb vector field $\xi_0$ are parameterized by the so-called \emph{normalised symplectic potentials}, the set of which we denote $\mathcal{S}_o(P_{\xi_0},\ell_1,\dots,\ell_d)$. These are strictly convex smooth functions on the interior of $P_{\xi_0}$ with a prescribed boundary behavior, see \eqref{eqSympPotdefn}, and all reaching their minimum on a fixed point $o\in P_{\xi_0}^\circ$.

The Sasaki structure $(D,J_\phi,\eta_0,\xi_0)$ is defined explicitly through the Hessian of $\varphi \in \mathcal{S}_o(P_{\xi_0},\ell_1,\dots,\ell_d)$, see Proposition~\ref{prop:toricSasaki_setup}. Note that the contact form $\eta_0$ is fixed, and does not depend on $\phi$. Applying directly the toric dictionary, we observe:
\begin{proposition}\label{prop:toricSasaki_intro}
    For $(\varphi,f)\in\mathcal{S}_o(P_{\xi_0},\ell_1,\dots,\ell_d)\times \mathcal{C}^\infty(P_{\xi_0},\mathbb{R}_{>0})$, the pseudo-hermitian structure $(D,J_\phi,f^{-1}\eta_0)$ is a cscTW if and only it solves 
    \begin{equation}\label{cscTWeqIntro}
        (\varphi^{ij}f^{-n-1})_{,ij}  = (n+1)f^{-n-2}\left(\varphi^{ij}f_{,ij}-C\right), 
    \end{equation}
    on $P_{\xi_0}^\circ$, where $(\varphi^{ij})$ is the inverse Hessian of $\varphi$ and $C$ is a constant, the value of the of Tanaka-Webster scalar curvature.
\end{proposition}
We will alternatively consider both $f$ and $\varphi$ as variables in equation \eqref{cscTWeqIntro}.
\begin{remark}
    The only $3$-dimensional toric contact manifold of Sasaki type is the $3$-sphere with its $\Torus \simeq S^1\times S^1$-invariant contact structure. In that case the polytope is a segment and the PDE \eqref{cscTWeqIntro} becomes an ODE, see \S\ref{ssCRyam2sphere}.
\end{remark}

\smallskip

For our main result, we consider paths in $\mathcal{S}_o(P_{\xi_0},\ell_1,\dots,\ell_d)$ of the form $\varphi_t= t^k \varphi_\infty + O(t^{k-1})$ as $t\to\infty$ for some smooth function $\varphi_\infty$ on $P$ and some $k>0$, where by $O(t^{k-1})$ we mean a collection of lower order terms. We consider two cases:
\begin{description}
    \item[C1] $\varphi_\infty$ is strictly convex. In this case, $\varphi_t\in\mathcal{S}_o(P_{\xi_0},\ell_1,\dots,\ell_d)$ for all $t\geq 0$;
    \item[C2] $\varphi_\infty$ is not convex, so that there is $T<\infty$ such that $\varphi_t\in\mathcal{S}_o(P_{\xi_0},\ell_1,\dots,\ell_d)$ for every $0\leq t < T$, and $\Hess\varphi_T$ is not invertible.
\end{description}
With this notation, our main result is the following.
\begin{theorem}\label{thm:boundary_behaviour}
    In case $1$, $\lim_{t\to +\infty} \CRYam(\varphi_t) \leq 0 $. In case $2$ instead, assume that there exists $x_T\in P^\circ$ and a unit vector $v_T$ such that
    \begin{equation}
        \langle v_T,(\Hess\varphi_T)_{x_T}v_T\rangle \in O(\abs{x-x_T}^{n-1+\alpha})
    \end{equation}
    for some $\alpha>0$. Then, $\lim_{t\to T^-} \CRYam(\varphi_t) = -\infty $.
\end{theorem}
Note that the additional assumption for case $2$ is always satisfied if $\dim N = 3$.

\begin{corollary}
    Given any compact toric contact manifold of Reeb type $(N,D,\Torus)$ there exists a $\Torus$-invariant CR structure on $D$ with negative CR Yamabe invariant.    
\end{corollary}

In particular, considering that constant scalar curvature Sasaki toric manifolds always have positive scalar curvature, we find a positive answer to a question posed in \cite{SungTakeuchi_CRYam}: we can construct (toric) pseudohermitian manifolds of arbitrary dimension $\geq 3$ such that the CR Yamabe invariant takes both positive and negative values on different (\textit{a posteriori}, non-equivalent) CR structures $J$ for the same contact manifold $(N,D)$. Note that any compact \emph{transversally Fano} (i.e.\ $c_1(D^{1,0})=0$ and, for one, equivalently any, Sasaki-Reeb vector field $\xi$, $c_1^{\xi}>0$) toric manifold admits a Sasaki-Einstein metric by Futaki-Ono-Wang \cite{FutakiOnoWang}. Transversally Fano Sasaki structures manifolds include all the Sasaki structures on the $U(1)$-principal bundle over Fano manifolds.       

\begin{corollary}
    Given any compact toric contact manifold of Reeb type $(N,D,\Torus)$ that is transversally Fano, there exist $\Torus$-invariant CR structures on $D$ with negative, zero and positive CR Yamabe invariants.    
\end{corollary}

The crucial ingredient in the proof of Theorem~\ref{thm:boundary_behaviour} is the following expression for the CR Einstein--Hilbert functional evaluated on $\Torus$-invariant pseudohermitian structures:
\begin{equation}\label{eq:EH_intro}
    \EH(f^{-1}\eta_0,J_\varphi) = \frac{2\int_{\partial P} f^{-n} d\sigma + n\int_P f^{-n-1}\,\Tr\left(H_\varphi\,\Hess f\right) dx}{\left(\int_P f^{-n-1} d\mu\right)^{\frac{n}{n+1}}},
\end{equation}
where $dx$ is the Lebesgue measure on $P=P_{\xi_0}$ and $d\sigma$ is a certain positive measure on the boundary of $P$ (see Section~\ref{sec:toric_yamabe} below for more details).

\smallskip

Theorem~\ref{thm:boundary_behaviour} suggests that it might be possible to compute the $\Torus$-invariant \emph{CR Yamabe invariant}
\begin{equation}
    \CRYam^\Torus(P_{\xi_0})\coloneqq\sup_{\varphi\in\mathcal{S}_o(P_{\xi_0},\ell_1,\dots,\ell_d)}\CRYam^\Torus(\varphi)
\end{equation}
by taking the supremum over a \emph{bounded} subset of $\mathcal{S}_o(P_{\xi_0},\ell_1,\dots,\ell_d)$, in $\mathcal{C}^0$-norm. As we recalled above, this supremum is in turn closely tied with the K-stability of the toric Sasaki manifold \cite{LahdiliLegendreScarpa_CRYam}.

\smallskip

Note however that there are at least two types of ``boundary points'' of the space of symplectic potentials $\mathcal{S}_o(P_{\xi_0},\ell_1,\dots,\ell_d)$ that are not included by Theorem~\ref{thm:boundary_behaviour}: first, $\varphi_\infty$ might be only \emph{weakly} convex. Second, one should allow the possibility of $\varphi_\infty$ not being smooth: indeed, the important class of \emph{geodesic rays} in $\mathcal{S}(P_{\xi_0},\ell_1,\dots,\ell_d)$ is given by piecewise-linear convex functions, see Section~\ref{sec:action}.

In particular, it seems crucial to extend the CR Yamabe energy to a completion of $\mathcal{S}_o(P_{\xi_0},\ell_1,\dots,\ell_d)$ that includes less regular convex functions. We start investigating this direction in Section~\ref{sec:action}, in which we show how to extend the CR Einstein--Hilbert functional on some singular pseudohermitian structures, following \cite{LahdiliLegendreScarpa_EHDF,LahdiliLegendreScarpa_CRYam}.
\begin{proposition}
    For any convex, piecewise-linear function $\varphi_\infty$, any affine-linear function $\ell$ and constant $a$ such that $\ell+a\varphi_\infty>0$,
    \begin{equation}\label{eq:EH_intro_ext}
        \EH(\ell+a\varphi_\infty , \varphi_P + \varphi_\infty) = \frac{2\int_{\partial P}(\ell+a\varphi_\infty)^{-n}d\sigma}{\left(\int_P(\ell+a\varphi_\infty)^{-n-1}dx\right)^{\frac{n}{n+1}}}.
    \end{equation}
\end{proposition}
The functional appearing on the right-hand side of \eqref{eq:EH_intro_ext} seems to be of special importance to understand the CR Yamabe invariant of the toric manifold, and is related to its K-stability. We refer to Section~\ref{sec:action} for more details and a few open questions in this direction.

\smallskip

In Section~\ref{sec:toric_yamabe} we also show an interesting result about cscS structures: they must lie at the boundary of the convex hull of cscTW structures.
\begin{proposition}
    Assume that $(N,D,J,\eta)$ is a toric cscS manifold and let $f_1, \dots, f_d$ be smooth $\Torus$-invariant positive functions such that $(N,D,J,f_i^{-1}\eta)$ is cscTW. Then, there exists no positive numbers $(\lambda_1, \dots, \lambda_d)$ such that $\sum_{i=1}^d \lambda_if_i =1$.
\end{proposition}

Finally, in Section~\ref{sec:examples} we show several examples of calculations of the CR Yamabe energy on toric manifolds and other manifolds of high symmetry, on which one can simplify the cscTW problem by imposing an appropriate ansatz. In view of our examples, we propose the following
\begin{question}\label{q:uniqueness}
     Let $(D,J,\eta, \chi)$ be a $\Torus$-invariant cscS structure with $\chi \in \kt$, such that $\EH(\eta,J)=\EHmin$. Is~$\eta$ a $\Torus$-invariant CR Yamabe minimiser in its conformal class $[\eta]^\Torus$? If so, is it the unique $\Torus$-invariant CR Yamabe minimiser in the conformal class?
\end{question}
Of course, it might be more interesting for this question to consider the case when $\Torus\subset\Aut(N,D,J,\eta,\chi)$ is \emph{maximal}, e.g.\ the toric case.

A positive answer to this question would constitute an extension of the Obata-type uniqueness result on CR manifolds (see \cite{JerisonLee_extremalsCRYam,LiWang_ObataCR}) beyond the Sasaki-Einstein case. In \S\ref{sec:trivialHirzebruch} and \S\ref{sec:nontoric}, we show in several examples that, without the assumption $\EH(\eta)=\EHmin$, the answer would be negative.

For a full summary of Section~\ref{sec:examples}, we refer the reader to the beginning of that section, but here we would like to highlight a couple of additional items:
\begin{itemize}
\item In \S\ref{sec:trivialHirzebruch} 
we construct some new explicit non-Sasaki cscTW examples over the trivial Hirzebruch surface, $\bbc\bbp^1\times \bbc\bbp^1$, that are not CR Yamabe minimizers.
\item In \S\ref{sec:CP1} and \S\ref{sec:proj_bundle}  we exhibit cases of  $\CRYam^{\bbt}(\varphi_a)$ becoming negative for explicit deformations, $\varphi_a$, of some standard potentials (or rather the inverse of their Hessians).
\end{itemize}

\subsection{Relation to other works and open problems}\label{sec:inv_yamabe}

An equivariant version of the classical Riemannian invariant has been first considered by Bergery \cite{Bergery_equiv}. Hebey and Vaugon \cite{HebeyVaugon_equiv} approached the equivariant Yamabe problem by conjecturing an extension of the well-known Aubin's criterion \cite{Aubin_yamabe-est} for the existence of minima of the Einstein--Hilbert action, and establishing this conjecture in several cases. We refer to \cite{Madani_equivconj1,Madani_equivconj2} for an overview of the equivariant Riemannian Yamabe problem and the Hebey--Vaugon conjecture. It is interesting to remark that in some cases, the equivariant Yamabe invariant is conjectured to coincide with the Yamabe invariant: see \cite{S1_Yamabe_equiv} for the case of the $U(1)$-Yamabe invariant on $S^3$.

For the \emph{CR} Yamabe problem, at version of the Hebey--Vaugon conjecture has been proposed in \cite{Ho_CRYam-equiv}. We refer to \cite{Afeltra_equiv} for some results on the equivariant CR Yamabe problem and an overview of the literature.

The main result of \cite{Ho_CRYam-equiv} can be stated as follows: if $G$ is a compact subgroup of CR automorphisms, denote by
\begin{equation}
    \Lambda_G\coloneqq \inf_{x\in N}\operatorname{card}\left\lbrace g.x \mid g\in G \right\rbrace.
\end{equation}
Then,
\(
    \CRYam^G(\eta) \leq \Lambda_G \, \CRYam(\mathbb{S}^{2n+1}),
\)
and if the inequality is strict, then there exists a solution of the $G$-equivariant CR Yamabe problem.

As in our toric case every orbit is positive-dimensional, it immediately follows that the toric CR Yamabe problem has a solution. This can also be seen as a by-product of the solution of the Yamabe problem on K-contact manifolds \cite{Zhang_Kcontact}, as we mentioned above. Even though we have this general existence result, it seems that even \emph{estimating} the CR Yamabe invariant is extremely complicated. To the best of our knowledge, Theorem~\ref{thm:boundary_behaviour} is the first result in this direction, and most questions remain open. For example, it is natural to wonder if $\CRYam^{\Torus}(\eta) = \CRYam(\eta)$, when $\eta$ is a $\Torus$-invariant Sasaki form.

\smallskip

Finally, we should highlight that, while we are chiefly interested in the existence of constant scalar curvature Sasaki structures, it would be interesting to see how to adapt our arguments to the more general class of $(p,q)$-Einstein-Hilbert functionals introduced in \cite{LahdiliLegendreScarpa_EHDF}, and to the classical Einstein--Hilbert functional in particular.

\subsection{Acknowledgments}

The idea for this paper emerged during the workshop ``Geometric Analysis in Lyon 2025'', and was in particular inspired by the mini-course on the CR Yamabe problem by Andrea Malchiodi. The authors would like to thank him for his interest in our paper and ideas for future directions. 

This work has benefited from a visit of C.T.-F.\ to the Université Lyon 1 covered by the Research and Innovation framework programme Horizon Europe {\normalsize\euflag} under grant agreement 10114932.

\section{Toric Sasaki manifolds}\label{sec:toric}

\subsection{Compact toric contact manifolds, a quick review}\label{ss:Review}

The classification of contact toric manifolds via their \emph{momentum cones} was achieved by Lerman~\cite{Lerman:contactToric}, incorporating the works of \cite{BanyagaMolino1} and \cite{BG:contactNOTE}, into a general statement and building on the compact toric symplectic orbifolds classification of \cite{LermanTolman}. In the case we are interested in, that is when there is a compatible toric Sasaki structure, the underlying contact toric manifold is necessarily of \emph{Reeb type}, in the sense that the Lie algebra of the torus contains a Reeb vector field.  Under this condition, the classification of Lerman is simplified and the momentum cone is a strictly convex cone (i.e.\ it lies in a half space) which determines the contact manifold uniquely \cite{BG:contactNOTE}. We give now a little more details, and introduce notation, on this theory.

\begin{definition}
    A compact \emph{contact toric manifold of Reeb type} $(N,D,\Torus)$ (of dimension $2n+1$) is a contact manifold $(N,D)$ of dimension $2n+1$, together with an effective contact action of a torus $\Torus$ of dimension $n+1$ whose Lie algebra $\kt$ contains a vector field $\xi \in \kt$ such that $\xi_x\notin D_x$, $\forall x\in N$.
\end{definition}
In other words, $\Torus$ is a subgroup of the group of contactormorphisms $\mbox{Con}(D)$ and the Reeb type condition implies that $(N,D)$ is co-oriented and the cone of Reeb field $\mbox{con}_+(D)$ meets $\kt$ non trivially. The \emph{Sasaki-Reeb cone}, relative to $\Torus$, is the polyhedral convex cone $\Reebcone\coloneqq\mbox{con}_+(D)\cap\kt$.

Recall that the symplectization, say $(D^0_+,\hat{\omega})$, of a co-oriented contact manifold is one of the two components of $D^0 \backslash N \subset T^*N$, where $D^0$ is the annihilator of $D$ in $T^*N$ together with the restriction $\hat{\omega}$ of the Liouville symplectic form on $T^*N$. It is a well-known fact that $(N,D)$ is contact and co-oriented if and only if $(D^0_+,\hat{\omega})$ is symplectic and that both objects determine the other uniquely. Therefore, a contact toric manifold corresponds to a (unique) \emph{toric symplectic cone}. The fact that it is a cone, in the sense that $\calL_{V} \hat{\omega}= 2\hat{\omega}$ for $V$, the radial vector field induced by the fiberwise $\bbr_+$-action on $D^0_+$, implies that $\hat{\omega}$ is exact (using the Cartan formula) and, thus, the naturally induced torus action on $D^0_+$ admits an equivariant momentum map $\hat{\mu}: D^0_+\ra \kt^*$, uniquely determined by the condition $\calL_{V} \hat{\mu} = 2\hat{\mu}$. The \emph{moment cone} of $(N,D,\Torus)$ is then the image of $\hat{\mu}$ in $\kt^*$, that we denote by $\calC\coloneqq\hat{\mu}(D^0_+)$. 

\begin{theorem}[\cite{Lerman:contactToric,BanyagaMolino1,BG:contactNOTE}]
    The moment cone of a compact contact toric manifold of Reeb type is a strictly convex cone in $\kt^*$ that does not contain $0$. Moreover, any two compact contact toric $\Torus$-manifolds with the same moment cone are $\Torus$-equivariantly contactomorphic.   
\end{theorem}

To complete the classification above, Lerman introduces the notion of \emph{good} cones, which are rational cones, with respect to the lattice of circle subgroups, say $\Lambda \subset \kt$, such that each face $F$ has primitive normals in its annihilator $F^0\subset \kt$ forming a $\bbz$-basis of $\Lambda\cap F^0$. Then, it follows from an argument in \cite{BanyagaMolino1}, see also \cite{Lerman:contactToric}, that, up to $\Torus$-equivariant contactomorphism, compact contact toric manifolds of Reeb type are classified by polyhedral strictly convex good cones up to $GL(\Lambda)$-linear maps. Moreover, they are all obtained by contact reduction of the standard contact toric sphere \cite{BG:contactNOTE}, and admit a compatible toric structure.  

\subsection{Toric Sasaki manifolds}\label{sss:ReviewSasaki}

Since the torus orbits in the symplectization are Lagrangians, on a toric contact manifold, say $(N,D,\Torus)$, the (fixed) torus $\Torus \subset \mbox{Con}(D)$ is maximal. Therefore, the Reeb vector field of any $\Torus$-invariant Sasaki structure lies in the Lie algebra $\kt$ of $\Torus$ and thus lies in $\Reebcone$. Conversely, any $\xi \in \Reebcone$ is the Reeb vector field of a Sasaki structure.

One way to see this is to use the main result of \cite{BG:contactNOTE} and find such Sasaki structure as a Sasaki reduction of the sphere. Another way is to use the symplectic potential setup developed, in the cone case, by Martelli-Sparks-Yau~\cite{MSYtoric} and Abreu~\cite{AbreuSasaki}, building on the correspondence between toric K\"ahler metrics on compact toric symplectic manifolds proved by Guillemin \cite{Guillemin_toric}.

To explain this correspondence and its Sasaki version, first notice that any $\xi\in\Reebcone$ determines a characteristic hyperplane $\Pi_\xi \coloneqq \{ x\in \kt^*\,|\, \langle x,\xi\rangle=1\}$ and a polytope
\[
    P_\xi := \Pi_\xi \cap \calC.
\]
We denote the primitive normals of $\calC$ by $\ell_1,\dots,\ell_d \in \Lambda\subset \kt$, so that
\[
    \calC= \left\lbrace x\in \kt^* \mid  \ell_i(x) \geq 0, i=1,\dots,d \right\rbrace.
\]
Note that the $\ell_i$'s are naturally identified with affine-linear functions on $\Pi_\xi$. Given a face $F$ of $P_\xi$ we denote by $F^\circ$ its interior and we let $C_{F\text{cvx}}^0(P_\xi, \bbr)$ the space of continuous functions $\phi$ such for any face $F$ of $P_\xi$, if $F^\circ\neq \emptyset$ then the restriction $\phi_{|_{F^\circ}}$ is smooth and convex. We introduce also the space of \emph{symplectic potentials} as follows
\begin{equation}\label{eqSympPotdefn}
    \calS(P_\xi,\ell_1,\dots,\ell_d ) \coloneqq \left\lbrace \phi \in C_{F\text{cvx}}^0(P_\xi, \bbr) \,\middle|
\begin{array}{l}
   \phi \in C^\infty(P^\circ_\xi, \bbr) \text{ convex}, \text{ and } \\
     \phi-\frac{1}{2}\sum_i\ell_i \log \ell_i  \in C^\infty(P_\xi, \bbr)  
\end{array}
\right\rbrace
\end{equation}
and observe that the group of affine-linear functions on $H_\xi$ acts on $\calS(P_\xi,\ell_1,\dots,\ell_d)$ by addition.  

By combining the works of \cite{Guillemin_toric,MSYtoric, AbreuSasaki} we have that
\begin{proposition}\label{prop:toricSasaki_setup}
    Toric Sasaki structures on $(N,D,T)$ with Reeb vector field $\xi$ are in one-to-one correspondence with the space of symplectic potentials modulo the action of the affine-linear functions $\mbox{Aff}(H_\xi,\bbr)$, that is $\calS(P_\xi,\ell_1,\dots,\ell_d )/\mbox{Aff}(\Pi_\xi,\bbr)$. Using action-angle coordinates $(x,\theta)\in P^\circ_\xi \times \kt/\bbr \xi$ the correspondence is explicit and the transversal K\"ahler metric is given by
    \[
        g_{\phi} = \sum_{i,j} G_{ij}(x)dx_i\otimes dx_j + H^{ij} d\theta_i \otimes d\theta_j,
    \]
    where $G=(G_{ij} =\phi_{,ij})_{1\leq i,j\leq n}$ is the Hessian of the symplectic potential $\phi \in C^\infty(P^\circ_\xi,\mathbb{R})$ and $H=(H^{ij})_{1\leq i,j\leq n}$ is the inverse of $G$.
\end{proposition}
Note that for each $\phi\in \calS(P_\xi,\ell_1,\dots,\ell_d)$ there is a CR-structure $J_\phi$ on $D$ determined by the Hessian of $\phi$. The fact that the induced CR-structure (complex structure in the symplectic/K\"ahler case) does not depend on the action-angle coordinates is addressed in \cite{HamFormsII}, see also \cite{Legendre_toricSasaki}.\\

As we will see below, most of the invariant riemannian operators (scalar curvature, laplacian etc) of a K\"ahler/Sasaki toric manifold depend only on the matrix valued function $H$ given by the inverse of the Hessian of a symplectic potential, it is sometimes useful to work directly with it instead of the potential itself. Another advantage is that this matrix valued function extends smoothly on the boundary of the polytope as we now recall. 

\begin{proposition}\cite{HamFormsII}\label{propH2F2_aboutH}
    A strictly convex smooth function $\phi$ on the interior, $\mathring{P}$, of $P$ is the restriction of a symplectic potential if and only if $(\Hess \phi)^{-1}$ is the restriction to $\mathring{P}$ of a smooth $S^2(\kt^*)$-valued function function $H : P \ra S^2(\kt^*)$ satisfying the following boundary conditions where $u_k=d\ell_k \in \kt$ and $F_k=\ell_k^{-1}(0)\cap P$: 
    \begin{itemize}
        \item $H_x(u_k,\cdot)=0$  for any $x\in \mathring{F}_k$,
        \item $dH(u_k,u_k)_x= 2u_k$  for any $x\in \mathring{F}_k$,
        \item $H_x$ is positive definite on $\kt/\langle u_{i_1},\dots,  u_{i_k}\rangle$ for any $x$ in the interior of the face $F_{i_1} \cap \dots, \cap F_{i_k}$.  
    \end{itemize}
\end{proposition}

\begin{remark}
    As explained in \cite[Appendix]{ApostolovCalderbankLegendre_weighted}, given a $\Torus$-invariant Sasaki manifold $(N,D_0,J_0,\eta,\xi)$, there is an equivariant bijection between the space of Sasaki structures compatibles with the fixed transversal holomorphic structure induced by $J_0$ on $TN/\bbr \xi$, up to transversal biholomorphisms, and the Teichm\"uller space of CR-structures on $D_0$ up to $\Torus$-equivariant contactomorphisms. In the toric case, this correspondence can be understood more directly using the beautiful Guillemin correspondence in toric K\"ahler geometry: {\it the Legendre transform of the symplectic potential is the K\"ahler potential} \cite{Guillemin_toric}. To apply this in the Sasaki setting, one needs to work directly on the toric K\"ahler cone associated to the toric Sasaki manifolds, this is done in \cite{MSYtoric}.
\end{remark}

\noindent {\bf We will use this theory as follows:} Let $(N,D,\Torus)$ be a compact contact toric manifold of Reeb type and pick $\xi\in\Reebcone$. This determines a moment cone $\calC$, a characteristic hyperplane $H_\xi$ and a polytope $P_\xi$, labeled by a set of affine-linear functions $\ell_1,\dots,\ell_d$. Note that since the contact structure $D$ is fixed, $\xi$ determines a unique contact form $\eta=\eta_\xi$. Given $\phi\in \calS(P_\xi,\ell_1,\dots,\ell_d)$ we consider the CR structure $(N,D,J_\phi)$, which is $\Torus$-invariant and admits a compatible Sasaki structure with with Reeb vector field $\xi$ and contact form $\eta$. Now, as the contact moment polytope is the orbit space (see e.g.\ \cite{Lerman:contactToric}), that is, $P_\xi \simeq N/\Torus$, the space of $\Torus$-invariant pseudohermitian structures on $(N,D,J_\phi)$ is in one-to-one correspondence with $C^\infty(P_\xi,\bbr_+)$. Indeed, each $f\in C^\infty(P_\xi,\bbr_+)$ determines a pseudo-hermitian structure $(N,D,J_\phi,f^{-1}\eta)$. 

\begin{corollary}
    Let $(N,D,\Torus)$ be a compact contact toric manifold of Reeb type and let $\xi\in\Reebcone$. Then the space of $\Torus$-invariant pseudo-hermitian structures is parameterized by $\calS(P_\xi,\ell_1,\dots,\ell_d )/\mbox{Aff}(H_\xi,\bbr) \times C^\infty(P_\xi,\bbr_+)$.   
\end{corollary}

\begin{remark}
    A different choice of Sasaki-Reeb vector field $\chi \in \Reebcone$ gives another parameter space $\calS(P_\chi,\ell_1,\dots,\ell_d )/\mbox{Aff}(H_\chi,\bbr) \times C^\infty(P_\chi,\bbr_+)$ but the map $ P_\xi \ni x \mapsto \frac{x}{\chi(x)} \in P_\chi$ identifies both spaces \cite{Legendre_toricSasaki}. 
\end{remark}

It will be useful at times to fix a coset for the action of $\mbox{Aff}(\Pi_\chi,\bbr)$ on the space of symplectic potentials.
\begin{definition}
    The space of \emph{normalised} symplectic potentials at $o\in P_{\xi_0}^\circ$ is
    \begin{equation} \label{eqNormSympPotdefn}
        \mathcal{S}_o(P_{\xi_0},\ell_1,\dots,\ell_d) = \left\lbrace \varphi\in\mathcal{S}(P_{\xi_0},\ell_1,\dots,\ell_d) \mid \varphi(o)=0,\ d\varphi(o)=0 \right\rbrace.
    \end{equation}
    Equivalently, $\mathcal{S}_o(P_{\xi_0},\ell_1,\dots,\ell_d)$ is the set of symplectic potentials that vanish at $o$ and have a minimum there.
\end{definition}

\begin{remark}\label{rem-linaffine}
    Our parametrisation of the conformal class of a contact form, $f^{-1}\eta$ for positive smooth functions $f$, has been chosen so that the Sasaki pseudohermitian forms in the conformal class are particularly simple to identify: if $\eta$ is sasakian, then $f^{-1}\eta$ is sasakian if and only if $f$ is a Killing potential for $(N,D,J,\eta)$. In particular in the toric case we see that $f^{-1}\eta$ is sasakian if and only if $f$ is an affine-linear positive function on $P$.
\end{remark}

\section{The CR Yamabe problem on toric CR manifolds}\label{sec:toric_yamabe}

\subsection{Scalar curvature of toric Sasaki-manifolds and the Tanaka-Webster scalar curvature}
As recalled above, for any Sasaki-Reeb vector field $\xi \in \Reebcone$ with transversal polytope $P=P_\xi$, we can pick a symplectic potential $\phi\in \mathcal{S}(P_\xi, \ell_1,\dots\ell_d)$ and consider the associated Sasaki structure $(N,D,J_\phi,\eta, \xi)$. The so-called {\it transversal scalar curvature} of $(N,D,J_\phi,\eta, \xi)$ \cite{BoyerGalicki} coincides in this case with the Tanaka-Webster scalar curvature, which in the toric case, is the pull-back by the moment map of the function \begin{equation}\label{eq:AbreuForm}
    \ScalTW(\eta, J_\phi) = -\sum_{i,j=1}^n (H^{ij})_{,ij}
\end{equation}
where again $(H^{ij}_\phi)$ is the unique continuous extension to $P$ of \[(\Hess \phi)^{-1} \in C^\infty (P^\circ, (\kt^*)^{\otimes 2}).\] As explained above, this matrix extends smoothly on $P$ and, thus, $\ScalTW(\eta,J_\phi)\in C^\infty(P,\bbr)$. The {\it Abreu formula} \eqref{eq:AbreuForm} is proved in \cite{Abreu_toric,AbreuSasaki}.

More generally, the Tanaka--Webster scalar curvature of the pseudo-hermitian structure $(N,D,J_\phi,f^{-1}\eta, \xi)$ for $(\phi,f)\in \calS(P_\xi,\ell_1,\dots,\ell_d ) \times C^\infty(P_\xi,\bbr_+)$ is
 \begin{equation}\label{eq:scalTW0}
    \ScalTW(f^{-1}\eta, J_\phi) = f\ScalTW(1,\phi)-2(n+1)\Delta^{g_\phi} f  - (n+1)(n+2) f^{-1} |df|^2_{g_\phi}.
\end{equation}
Using action-angle coordinates $(x,\theta)$ as in \S\ref{sss:ReviewSasaki}, any $\Torus$-invariant function on $N$ is a function of $x$, so that $f$, $\Delta^{g_\phi} f = - (H^{ij}_\phi f_i)_j$ and $|df|^2_{g_\phi} = H_{\phi}^{ij}f_if_j$ are functions in the variable $x$, i.e.\ functions on $P$. Expressing \eqref{eq:scalTW0} in action-angle coordinates gives the following expression for the scalar curvature 
\begin{equation}\label{eq:scalTW}
\begin{split}
    \ScalTW(\phi,f)\coloneqq & \ScalTW(f^{-1}\eta, J_\phi)\\
    = & -(H_\phi^{ij})_{,ij}f +2(n+1)(H_\phi^{ij}f_i)_j -(n+1)(n+2) f^{-1} H_\phi^{ij}f_if_j.
\end{split}
\end{equation}
We use the notation $\ScalTW(\phi,f) \coloneqq \ScalTW(f^{-1}\eta, J_\phi)$ for the rest of this paper (so that \eqref{eq:AbreuForm} is $\ScalTW(\phi,1)$).

\subsection{The equivariant CR Yamabe problem over a toric Sasaki manifold}

For notational convenience, we fix once and for all a Sasaki-Reeb vector field $\xi_0$ and a labelling $\ell_1,\dots,\ell_d$ of the polytope, and set $P=P_{\xi_0}$, $\mathcal{S}(P)=\mathcal{S}(P_{\xi_0},\ell_1,\dots,\ell_d)$. 

As any smooth $\Torus$-invariant function on $N$ is the pullback of a smooth function on (an open set containing) $P$, the $\Torus$-equivariant cscTW equation translates to a PDE on functions defined on $P$: it consists in finding a function $f\in\mathcal{C}^{\infty}(P,\mathbb{R}_{>0})$ such that
\begin{equation}\label{eq:toric_cscTW}
    C= -(H_\phi^{ij})_{,ij}f +2(n+1)(H_\phi^{ij}f_i)_j -(n+1)(n+2) f^{-1} H_\phi^{ij}f_if_j.
\end{equation}
for some constant $C$. In this perspective, the equivariant cscTW problem over a compact CR manifold of toric Sasaki type is equivalent to a \emph{boundary value problem} using the boundary conditions imposed on $(H_{ij})$ see Proposition~\ref{propH2F2_aboutH} which is particularly simple to express on the toric Sasaki $3$-sphere as we now show.

\subsubsection{The toric cscTW on the 3-sphere as a boundary value problem.} \label{ssCRyam2sphere}
In the case of the toric $3$-sphere, the polytope is a segment $P=[-1,1]$ and any toric Sasaki structures $(D_H,J_H,\eta_H)$ corresponds via Proposition~\ref{propH2F2_aboutH} and the theory explained above, to a smooth function $H:P\ra \bbr$ positive on the interior $(-1,1)$, vanishing at the end points and such that $H'(\pm 1)= \mp 2$. The cscTW equation \eqref{eq:toric_cscTW} in the conformal class of $\eta_H$ corresponds to a solution $f\in C^\infty(P,\bbr_+)$ of the equation   
   \begin{equation}\label{eq:toric_cscTWsphere}
    C= -H''f + 4 H' f' +4Hf'' - 6 f^{-1} H (f')^2.
\end{equation} where $C\in\bbr$ is a constant.  The boundary conditions of Proposition~\ref{propH2F2_aboutH} impose that  \begin{equation}\label{eq:toric_cscTWspherebd1}
    C= -H''(-1) f(-1) + 8 f'(-1) \mbox{ and }  C= -H''(1) f(1) - 8 f'(1). 
\end{equation}

\smallskip

In this setting, the CR Einstein--Hilbert functional is
    \begin{equation}
        \EH(\varphi,f)=\frac{2\int_{\partial P}f^{-1}d\sigma+\int_Pf^{-2}f''H dx}{\norm{f^{-1}}_{L^2(P)}}.
    \end{equation}
A minimizer $f$ of $\EH(\varphi,f)$ with respect to $H$ solves the toric cscTW equation~\eqref{eq:toric_cscTW} in dimension $1$, namely \eqref{eq:toric_cscTWsphere}. Recall that $C$ is the value of $\EH(\varphi,f)$, up to multiplication by a positive constant. Now, let $x\in P^\circ$ be a critical point of $H : [-1,1] \ra \bbr_{\geq 0}$. Evaluating at $x$ we obtain
\begin{equation}
    C+f(x)\,H''(x)+6 f(x)^{-1} H(x)(f'(x))^2= 4 H(x) f''(x),
\end{equation}
and as both $f$ and $H$ are positive,
\begin{equation}
    C+f(x)\,H''(x) \leq 4 H(x) f''(x).
\end{equation}
Thus we see that in general, it might not be possible for $f$ to be concave down: for example, if $\CRYam(H)\geq 0$ and the critical point $x$ is a local minimum of $H$, if follows that
\[
    4 H(x) f''(x) \geq C + f(x) \, H''(x) \geq 0.
\]
Note that we will continue to look at the $3$-sphere (over $\bbc\bbp^1$) in \S\ref{sec:CP1}.

\subsubsection{Link with weighted cscK structures} In the notation of Lahdili \cite{Lahdili_weighted}, for any function $\v\in C^\infty(P,\bbr_{>0})$, the $\v$-\emph{weighted scalar curvature} is
\begin{equation}
    \scal_\v(\omega,J_\varphi) \coloneqq -\sum_{i,j=1}^n (\v H^{ij})_{,ij}.
\end{equation}
\begin{lemma}
    For $\v(x) = f^{-n-1}(x)$ and a K\"ahler form $\omega_0$ such that $\pi^*\omega_0=d\eta_0$, 
    \begin{equation}\label{scalTWvsWeightedcscK}
        \ScalTW(\varphi,f) = f^{n+2}\scal_\v(\omega_0,J_\varphi) + (n+1) \langle H,\Hess f\rangle 
    \end{equation}
    where $\langle H,\Hess f\rangle = \Tr(H\,\Hess f) = \sum_{i,j=1}^n H^{ij}f_{ij}.$     
\end{lemma}
\begin{proof}
    The direct computation gives 
    \begin{equation}
    \begin{split}
        (H^{ij}f^{-n-1})_{,ij} &= (H^{ij}_{,i}f^{-n-1})_{,j} -(n+1)(f^{-n-2}H^{ij}f_{,i})_{,j}  \\
        &= H^{ij}_{,ij}f^{-n-1} -(n+1)f^{-n-2}H^{ij}_{,i}f_{,j} -(n+1)f^{-n-2}H^{ij}_{,j}f_{,i} \\
        &\qquad\qquad\qquad -(n+1)f^{-n-2}H^{ij}f_{,ij} +(n+1)(n+2)f^{-n-3}H^{ij}f_{,i}f_{,j}\\
        &= - f^{-n-1}\scal(\omega) + 2(n+1)f^{-n-2} \Delta^\omega f + 2(n+1) f^{-n-2} H^{ij} f_{,ij}\\
        &\qquad\qquad\qquad -(n+1)f^{-n-2}H^{ij}f_{,ij}  +(n+1)(n+2) f^{-n-3}\abs{df}^2_\omega\\
        &= - f^{-n-2}\ScalTW(\varphi,f) + (n+1)f^{-n-2}H^{ij}f_{,ij}
    \end{split}
    \end{equation}
    since $H^{ij}f_{,i}f_{,j}=\abs{df}^2_{\omega_0}$, and 
    \begin{equation}
        \Delta^{\omega_0} f=-\sum_{i,j} (H^{ij} f_{,j})_{,i} = -\sum_{i,j} (H^{ij}_{,i} f_{,j} +H^{ij} f_{,ij}).\qedhere
    \end{equation}  
\end{proof}
In particular, we can also rewrite the cscTW equation as
\begin{equation}
    (H^{ij}f^{-n-1})_{,ij}  = (n+1)f^{-n-2}\left(H^{ij}f_{,ij}-C\right). 
\end{equation}
We will alternatively consider both $f$ and $\varphi$ as variables in this equation.

\smallskip

To rewrite the toric CR Einstein--Hilbert functional, we will need the following integration by parts Lemma.
\begin{lemma}[\cite{Donaldson_stability_toric,ApostolovMaschler_EinsteinMaxwell}]
    For any functions $\psi_0,\psi_1 \in C^\infty(P,\bbr)$ and any symplectic potential $\varphi$, we have
    \begin{equation}
        \int_P \psi_1 (H^{ij}\psi_0)_{ij} dx = -2\int_{\partial P} \psi_1\psi_0 \,d\sigma  +\int_P\psi_0 \langle H, \Hess \psi_1\rangle dx 
    \end{equation}
    where $d\sigma$ is the measure on $\partial P$ defined by $d\ell_i\wedge d\sigma =-dx$ on the facet $\ell_i^{-1}(0)\cap P$.
\end{lemma}
Using this Lemma with $\psi_0 =f^{-n-1}=\v$ and $\psi_1=f$, we can rewrite
\begin{equation}
    \Scaltot(\phi,f) = \int_P f^{-n-1}\ScalTW(\phi,f) dx.
\end{equation}
\begin{corollary}
    With the previous notation, we have the following expressions for the volume and the total scalar curvature of toric pseudohermitian forms:
    \begin{equation}
    \begin{split}
        \Voltot(\varphi,f) =& \Voltot(f) =  \int_P f^{-n-1} dx \\
        \Scaltot(\varphi,f) =& 2\int_{\partial P}f^{-n}d\sigma +n \int_P f^{-n-1} \langle H_\varphi, \Hess f\rangle dx.
    \end{split}
    \end{equation}
\end{corollary}
Hence, in this setting the CR Einstein--Hilbert functional becomes
\begin{equation} \label{toricEH}
    \EH(\varphi,f)=\frac{2\int_{\partial P}f^{-n}d\sigma +n \int_P f^{-n-1} \langle H, \mbox{Hess} f\rangle dx}{\left(\int_P f^{-n-1} dx\right)^{\frac{n}{n+1}}}
\end{equation}

\begin{remark}
    When $f$ is an affine linear-function, that is when $(D,J_{\phi},f^{-1}\eta)$ is Sasaki (see Remark~\ref{rem-linaffine}), we get 
    \begin{equation}\label{toricEHaff}
        \EH(\varphi,f)=\frac{2\int_{\partial P}f^{-n}d\sigma }{\left(\int_P f^{-n-1} dx\right)^{\frac{n}{n+1}}}
    \end{equation}
    which is independent of the symplectic potential as expected for Sasaki structure see \cite{FutakiOnoWang}. The minimum of the Einstein-Hilbert function on the Sasaki-Reeb vector field is known to exist, see \cite{BoyerHuangLegendre_DH}, and is then
$$\EH_{\min} := \inf_{\substack{f\in \operatorname{Aff}(\kt^*,\bbr),\\ f>0 \mbox{ on } P }} \frac{2\int_{\partial P}f^{-n}d\sigma }{\left(\int_P f^{-n-1} dx\right)^{\frac{n}{n+1}}}$$ in the toric case. 
\end{remark}

As $f>0$, the boundary term in \eqref{toricEH} is always positive.
\begin{corollary}
    If $\ScalTW(\varphi,f)$ is a non-positive constant then $f$ is not convex.    
\end{corollary}
One might wonder if the toric CR Einstein--Hilbert functional is always minimised by a \emph{concave} function $f$, as in this case the polytope integral in the numerator of~\eqref{toricEH} is negative. The discussion in \S\ref{ssCRyam2sphere} shows that this is not the case, in general.

\subsubsection{The convex hull of cscTW structures.} The toric setting makes it possible to apply arguments from convex geometry to the cscTW equation. For example, the proof of \cite[Lemma 5.2]{BHLT25} extends straightforwardly to show the following claim which implies that the convex hull of cscTW contains no cscS. 
\begin{lemma}
    Assume that $(N,D,J,\eta,\xi)$ is a cscS toric manifold and let $f_1, \dots, f_k$ be any smooth $\Torus$-invariant positive functions such that $(N,D,J,f_i^{-1}\eta)$ is cscTW. Then, there exists no positive numbers $(\lambda_1, \dots, \lambda_k)$ such that $\sum_{i=1}^k \lambda_i f_i =1$.     
\end{lemma}

\begin{proof}
    The argument is exactly the same of \cite[Lemma 5.2]{BHLT25} that we recall for the reader's convenience. Let $P$ be the transversal polytope associated to $\xi$ (this is where we use that $(N,D,J,\eta,\xi)$ is Sasaki not only pseudo-hermitian). Then by $\Torus$ invariance (this is where we use the toric assumption) the functions $f_1,\dots, f_k$ are actually pull-back of functions on $P$, that is $f_m= h_m\circ \mu$ for some $ h_1,\dots, h_m\in C^\infty(P,\bbr_{>0})$. Assume that these functions are like in the claim, so correspond to cscTW structures on the same CR-structure $(N,D,J)=(N,D,J_\phi)$ where $\phi \in \mathcal{S}(P)$. Let $(\lambda_1, \dots, \lambda_k)$ be such that $\sum_{m=1}^k \lambda_m f_m =1$ and assume, for simplicity that $\ScalTW(\phi,1) =1$. Then for $m=1,\dots, k$, $f_m$ satisfies the equation \eqref{eq:scalTW0} with the left hand side being a constant, say $C_m$, 
    \begin{equation}
        C_m= f_m-2(n+1)\Delta^{g_\phi} f_m  - (n+1)(n+2) f_m^{-1} |df_m|^2_{g_\phi}.
    \end{equation}
    Taking the weighted sum of these equations we get 
    \begin{equation}
        \sum_m \lambda_m C_m= 1 -0 - (n+1)(n+2) \sum_m \lambda_m f_m^{-1} |df_m|^2_{g_\phi}.
    \end{equation}
    Each fixed point of $\Torus$ is a common critical point of the functions $f_1,\dots, f_k$ since $d\mu=0$ on these points. Therefore $\sum_m C_m= 1 $ which, in turn, implies that 
    $$\sum_m \frac{\lambda_m}{f_m} |df_m|^2_{g_\phi}\equiv 0$$ on $N$. But for $m=1,\dots, k$, $\lambda_m>0$ so $df_m\equiv 0$.
\end{proof}

\subsection{Variational formulas in the toric setting}

The gradient of the $\EH$ functional is easily computed, and can be found in several standard references, cf.\ \cite{JerisonLee_CRYamabe,EinsteinHilbert-SasakiFutaki}. In this section, we recall the main formulas.

We consider a $\mathcal{C}^1$ path of symplectic potentials $\phi_t \in \calS(P,\ell_1,\dots,\ell_d)$ and positive functions $f_t \in \mathcal{C}^\infty(P,\bbr_{>0})$, where $t$ is a real variable, taking values in some interval of $\bbr$ containing $0$.  For ease of notation, we let
\[
    G_t \coloneqq \Hess(\phi_t),\;\; H_t \coloneqq \Hess(\phi_t)^{-1}
\]
\[
    \scal_t \coloneqq \ScalTW(\varphi_t,f_t), \;\;
    \Scaltot_t \coloneqq \Scaltot(\varphi_t,f_t) \; \mbox{ and } \;
    \Voltot_t \coloneqq \Voltot(f_t).
\]
Using the fact that the only term in $\EH(\phi_t,f_t)$ depending on $\phi_t$ is the second one (see \eqref{toricEH}), namely $n \int_P f^{-n-1} \langle H, \mbox{Hess} f\rangle dx$, and the relation
\[
    \frac{d}{dt} H_t = -H_t \left(\frac{d}{dt}G_t \right) H_t,
\]
we get
\begin{equation}
\begin{split}
    \frac{d}{dt} \EH(\phi_t,f_t) = & \frac{\partial}{\partial(f_t)} \EH(\phi_t,f_t) + \frac{\partial}{\partial(\phi_t)} \EH(\phi_t,f_t) \\
    = & \frac{-n}{\bfV_t^{n/n+1}}\int_P \frac{\dot{f}_t}{f_t^{n+2}}\left(\scal_t -\frac{\bfS_t}{\bfV_t}\right) dx \\
    & + \frac{-n}{\bfV_t^{n/n+1}} \int_P f^{-1-n}\langle H_t \Hess \dot{\phi}_t,H_t \Hess f_t\rangle dx.
\end{split}
\end{equation}
In particular, we obtain the following
\begin{theorem}\label{theocritEH=cscS}[\cite{AfeltraMalchiodi_EH,LahdiliLegendreScarpa_EHDF}]
    The critical points of $\EH(\varphi,f)$ are characterised by the conditions
    \begin{equation}
    \begin{cases}
        \ScalTW(\varphi,f)=\text{const.}\\
        \Hess(f) = 0.
    \end{cases}
    \end{equation}
    In other words, the critical points are Sasaki structures of constant transversal scalar curvature.
\end{theorem}

\section{The toric CR Yamabe energy}\label{sec:theorem1.1}

In this Section, we prove Theorem~\ref{thm:boundary_behaviour}. Along the way, we will identify another functional that governs the behaviour of $\CRYam^T$ at the ``boundary at infinity'' of $\mathcal{S}(P)$. Theorem~\ref{thm:boundary_behaviour} consists of two parts, since we have to analyse the behaviour of $\CRYam^\Torus$ along two types of paths in $\mathcal{S}(P)$. See Figure~\ref{fig:picture_symplpot} for reference and for the sketch of the proof. Recall that $\EH(\varphi,f)$ is invariant under the addition of an affine-linear function to $\varphi$, hence $\CRYam^\Torus$ is invariant under this action as well; hence, in this Section we fix a point $o\in P^\circ$ and work with normalised symplectic potentials $\varphi\in\mathcal{S}_o(P)$.

To give an estimate for the CR Yamabe invariant, we evaluate $\EH$ on appropriate test functions. It is important to remark that, for a smooth symplectic potential~$\varphi$, $\EH(\varphi,f)$ is well-defined for any continuous, positive $f$ such that $\Hess f$ is a bounded distribution on $P$. Moreover, in the definition of $\CRYam^{\Torus}$, one can take the infimum over the set of all such functions, not necessarily smooth ones.

\begin{figure}[t]
\centering
\resizebox{0.8\textwidth}{!}{%
\begin{circuitikz}
\tikzstyle{every node}=[font=\fontsize{11.9pt}{15.5pt}\selectfont]
\draw [short] (12.25,18.875) -- (9.25,15);
\draw [short] (9.25,15) -- (9.25,9.625);
\draw [short] (9.25,9.625) -- (20.000,9.625);
\draw [dashed] (14.125,18.375) -- (11.625,15);
\draw [dashed] (11.625,15) .. controls (11,14.125) and (10.875,14) .. (11,13);
\draw [dashed] (11,13.25) -- (11,12.25);
\draw [dashed] (11,12.375) .. controls (11.125,11.125) and (11.625,11.375) .. (12.5,11.375);
\draw [dashed] (12.375,11.375) -- (21,11.375);
\node [font=\fontsize{27.9pt}{36.2pt}\selectfont, fill={rgb,255:red,255; green,255; blue,255}, fill opacity=1, text opacity=1, inner xsep=0.080cm, inner ysep=0.085cm, rounded corners=0.020cm] at (13.75,14.375) {$.$};
\node [font=\fontsize{11.9pt}{15.5pt}\selectfont, fill={rgb,255:red,255; green,255; blue,255}, fill opacity=1, text opacity=1, inner xsep=0.080cm, inner ysep=0.085cm, rounded corners=0.020cm] at (14.125,14.375) {$\varphi_0$};
\node [font=\fontsize{11.9pt}{15.5pt}\selectfont, fill={rgb,255:red,255; green,255; blue,255}, fill opacity=1, text opacity=1, inner xsep=0.080cm, inner ysep=0.085cm, rounded corners=0.020cm] at (18.000,15.000) {$\mathcal{S}(P)$};
\node [font=\fontsize{11.9pt}{15.5pt}\selectfont, fill={rgb,255:red,255; green,255; blue,255}, fill opacity=1, text opacity=1, inner xsep=0.080cm, inner ysep=0.085cm, rounded corners=0.020cm] at (22.125,16.000) {$\CRYam\leq 0$};
\node [font=\fontsize{11.9pt}{15.5pt}\selectfont, fill={rgb,255:red,255; green,255; blue,255}, fill opacity=1, text opacity=1, inner xsep=0.080cm, inner ysep=0.085cm, rounded corners=0.020cm] at (14.125,16.000) {$\CRYam\geq 0$};
\node [font=\fontsize{11.9pt}{15.5pt}\selectfont, fill={rgb,255:red,255; green,255; blue,255}, fill opacity=1, text opacity=1, inner xsep=0.080cm, inner ysep=0.085cm, rounded corners=0.020cm] at (10.75,10.50) {$\CRYam<0$};
\draw [line width=0.5pt, -{Stealth[scale=1.5]}, ] (13.75,14) -- (14.375,9.75)node[pos=0.45,right, fill={rgb,255:red,255; green,255; blue,255}, fill opacity=1, text opacity=1, inner xsep=0.080cm, inner ysep=0.085cm, rounded corners=0.020cm]{$\varphi_0+t\varphi_\infty$};
\node [font=\fontsize{26.7pt}{34.8pt}\selectfont, color={rgb,255:red,255; green,0; blue,0}, text opacity=1, , fill={rgb,255:red,255; green,255; blue,255}, fill opacity=1, text opacity=1, inner xsep=0.080cm, inner ysep=0.085cm, rounded corners=0.020cm] at (14.375,9.550) {$.$};
\draw [ line width=0.6pt ] (14.625,13.125) circle (7.25cm);
\node [font=\fontsize{11.9pt}{15.5pt}\selectfont, fill={rgb,255:red,255; green,255; blue,255}, fill opacity=1, text opacity=1, inner xsep=0.080cm, inner ysep=0.085cm, rounded corners=0.020cm] at (9.375,8.5) {$\mathcal{C}^{\infty}(P)$};
\node [font=\fontsize{27.9pt}{36.2pt}\selectfont, fill={rgb,255:red,255; green,255; blue,255}, fill opacity=1, text opacity=1, inner xsep=0.080cm, inner ysep=0.085cm, rounded corners=0.020cm] at (14.950,5.900) {$.$};
\node [font=\fontsize{11.9pt}{15.5pt}\selectfont, fill={rgb,255:red,255; green,255; blue,255}, fill opacity=1, text opacity=1, inner xsep=0.080cm, inner ysep=0.085cm, rounded corners=0.020cm] at (15.25,5.25) {$\varphi_{\infty}$ not convex};
\node [font=\fontsize{11.9pt}{15.5pt}\selectfont, fill={rgb,255:red,255; green,255; blue,255}, fill opacity=1, text opacity=1, inner xsep=0.080cm, inner ysep=0.085cm, rounded corners=0.020cm] at (21.050,18.400) {$\varphi_{\infty}$ strictly convex};
\node [font=\fontsize{27.9pt}{36.2pt}\selectfont, fill={rgb,255:red,255; green,255; blue,255}, fill opacity=1, text opacity=1, inner xsep=0.080cm, inner ysep=0.085cm, rounded corners=0.020cm] at (19.940,18) {$.$};
\node [font=\fontsize{11.9pt}{15.5pt}\selectfont, fill={rgb,255:red,255; green,255; blue,255}, fill opacity=1, text opacity=1, inner xsep=0.080cm, inner ysep=0.085cm, rounded corners=0.020cm] at (14.375,9.25) {$\varphi_T$};
\node [font=\fontsize{11.9pt}{15.5pt}\selectfont, fill={rgb,255:red,255; green,255; blue,255}, fill opacity=1, text opacity=1, inner xsep=0.080cm, inner ysep=0.085cm, rounded corners=0.020cm] at (17.000,9.250) {$\CRYam=-\infty$};
\node [font=\fontsize{11.9pt}{15.5pt}\selectfont, fill={rgb,255:red,255; green,255; blue,255}, fill opacity=1, text opacity=1, inner xsep=0.080cm, inner ysep=0.085cm, rounded corners=0.020cm] at (18.000,11.000) {$\CRYam=0$};
\draw [line width=0.6pt, dashed] (14.5,8.875) -- (14.875,6.375);
\end{circuitikz}
}%
\caption{\small Paths in the space of symplectic potentials $\mathcal{S}(P)\subset\mathcal{C}^\infty(P^\circ)$ and behaviour of the CR Yamabe function.}
\label{fig:picture_symplpot}
\end{figure}
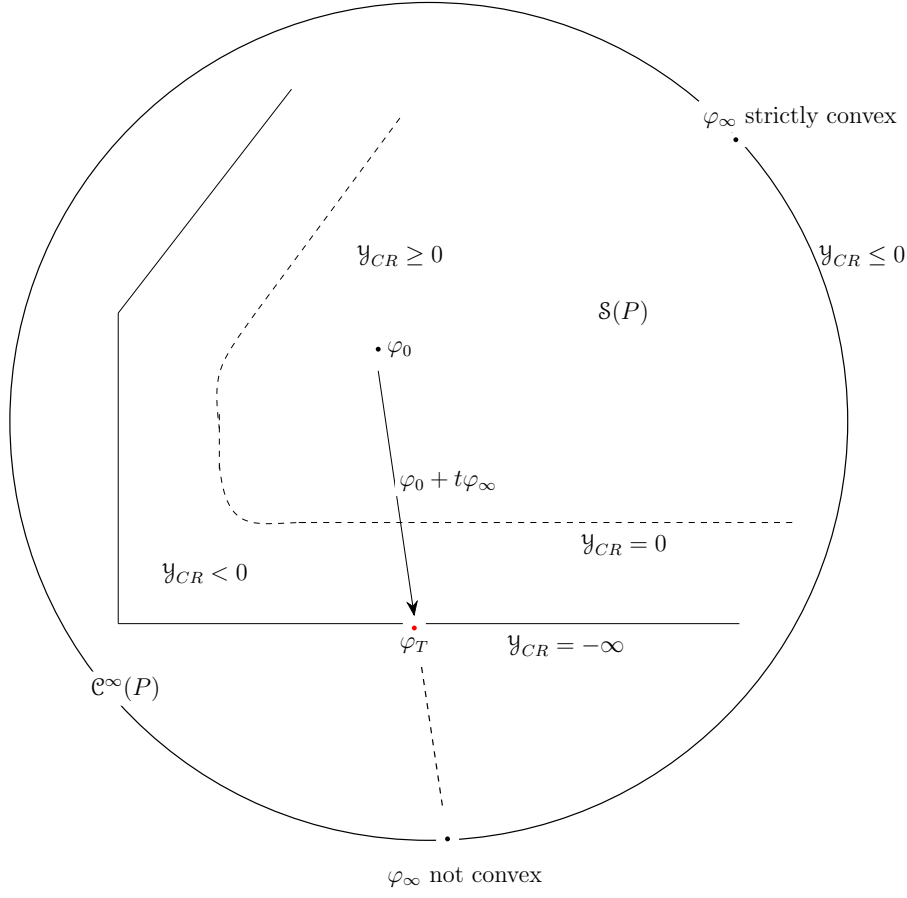

\noindent \textbf{Proof of Theorem~\ref{thm:boundary_behaviour}, case 2.} \quad Assume that $\varphi_t$ is a path of normalised symplectic potentials for $t<T$, and that there exists $x_T\in P^{\circ}$ for which $(\Hess\varphi_T)(x_T)$ is not invertible. As $\varphi_t$ is a path of symplectic potentials, for every $t<T$ the matrix $\Hess\varphi_t$ is positive definite. This means that $(\Hess\varphi_T)(x_T)$ has an eigenvector $\nu$ of eigenvalue zero. We can choose a system of coordinates in $\mathbb{R}^n$ such that $\nu=e_1$ and $x_T=0\in\mathbb{R}^n$. Let then $f_{pl}$ be the piecewise-linear function
\begin{equation}\label{eq:piecewise_test}
    f_{pl} = \begin{cases}
                m_1 x_1 + a & \text{if }x\leq 0\\
                m_2 x_1 + a & \text{if }x\geq 0,
            \end{cases}
\end{equation}
and we compute the total scalar curvature as follows: for simplicity, let $G_t\coloneqq\Hess\varphi_t$ and $H_t\coloneqq G_t^{-1}$.
\begin{equation}\label{eq:scaltot_piecewiselin}
\begin{split}
    \Scaltot(f_{pl},\varphi_t) =& 2\int_{\partial P}f_{pl}^{-n}d\sigma + \int_P f_{pl}^{-n-1} \langle H_t,\Hess f_{pl}\rangle d\mu\\
    =& 2\int_{\partial P}f_{pl}^{-n}d\sigma + \frac{m_2-m_1}{a^{n+1}}\int_{P\cap\{x_1=0\}} H^{11}_t d\mu.
\end{split}
\end{equation}
We choose $m_2<m_1$, so that $f_{pl}$ is concave and the second term is negative. At this point, we would like to show that as $t\to T$, the integral in the second term becomes arbitrarily large. We know that as $t\to T$, $\langle e_1,G_t(0) e_1\rangle\to 0^+$, and as $H$ is the inverse of $G$,
\begin{equation}\label{eq:H11_div}
    H^{11}_t(0) = \langle e_1,H_t(0) e_1\rangle \xrightarrow{t\to T} + \infty.
\end{equation}
If $\dim N=3$, so that $\dim P=1$, this is enough to conclude that the total scalar curvature diverges to $-\infty$ as $t\to T$. As the test function $f_{pl}$ does not depend on $t$, we deduce that in this case
\begin{equation}
    \lim_{t\to T^-}\EH(f_{pl},\varphi_t)=-\infty.
\end{equation}

However, if $\dim P>1$, equation \eqref{eq:H11_div} does not necessarily mean that the integral
\begin{equation}
    \int_{P\cap\{x_1=0\}} H^{11}_t d\mu
\end{equation}
diverges as $t\to T$. For example, one has to exclude the possibility that
\begin{equation}
    H^{11}_t(x)=\frac{T-t}{(T-t)^{1+\varepsilon}+\lvert x\rvert^2},
\end{equation}
In which case the integral of $H^{11}_t$ vanishes (for $\varepsilon<1$), as $t\to T^-$. Part $(2)$ of Theorem~\ref{thm:boundary_behaviour} is a consequence of the following.

\begin{lemma}\label{lemma:Gv_small}
    Suppose that $\varphi_t$ is a path of normalised symplectic potentials, and that there exist a vector $v_t$ of unit norm, $\delta>0$ and $r>0$ such that
    \begin{equation}\label{eq:Gv_small}
        \langle v_t, G_t(x)v_t \rangle < \delta
    \end{equation}
    for every $x\in \{\ell_t =0\} \cap B_r(x_0)$, where $\ell_t(x)\coloneqq\langle x-x_0 ,v_t\rangle$ and $B_r(x_0)$ is a ball of radius $r$ inside $P$. Consider the test functions
    \begin{equation}
        f_t(x) \coloneqq \begin{cases}
            m_1\langle x,v_t\rangle - a_t +c_t &\text{ if } \ell_t(x) \leq 0\\
            m_2\langle x,v_t\rangle + a_t +c_t &\text{ if } \ell_t(x) \geq 0,
        \end{cases} \quad a_t\coloneqq \frac{1}{2}(m_2-m_1)\langle x_0,v_t\rangle,
    \end{equation}
    where $c_t$ is a constant, chosen so that $f_t>0$ on $P$. Then, if we fix $m_2<m_1$ we can choose $c_t$ to be constant in $t$, and there exists a positive constant $k$, not depending on $t$ or $\delta$, such that
    \begin{equation}
        \int_P f_t^{-n-1}\langle H_t, \Hess f_t\rangle dx \leq 
        (m_2-m_1)k\,\delta^{-1}\operatorname{Vol}(\{\ell_t=0\}\cap B_r(x_0)).
    \end{equation}
\end{lemma}
Note that the result can be improved somewhat, for example \eqref{eq:Gv_small} could be substituted by an integral inequality rather than a pointwise one.
\begin{proof}
    The value of $f_t$ along the ``crease'' $\{\ell_t=0\}$ is $\frac{1}{2}(m_1+m_2)\langle x_0,v_t\rangle+c_t$, while its Hessian is (since $v_t$ has unit norm)
    \begin{equation}
        \Hess f_t = (m_2-m_1) \delta_{\{\ell_t=0\}\cap P} \, v_t \otimes v_t^\transpose.
    \end{equation}
    Hence, for the integral we obtain
    \begin{equation}
        \int_P f_t^{-n-1} \langle H_t,\Hess f_t\rangle dx = \frac{(m_2-m_1)}{\left(\frac{1}{2}(m_1+m_2)\langle x_0,v_t\rangle+c_t\right)^{n+1}} \int_{\{\ell_t=0\}} v_t^{\transpose} H_t v_t\, dx.
    \end{equation}
    As $H_t$ is positive-definite in $P^\circ$, we can restrict the integral to the ball to obtain
    \begin{equation}
        \int_{\{\ell_t=0\}\cap P} v_t^{\transpose} H_t v_t\, dx \geq \int_{\{\ell_t=0\}\cap B_r(x_0)} v_t^{\transpose} H_t v_t\, dx.
    \end{equation}
    Now, we claim that, for every $x\in B_{x_0}(r)$,
    \begin{equation}\label{eq:vHv}
        v_t^\transpose H_t(x) v_t > \delta^{-1}.
    \end{equation}
    Assuming \eqref{eq:vHv} for the moment, one gets
    \begin{equation}
        \int_{\{\ell_t=0\}\cap P} v_t^{\transpose} H_t v_t\, dx \geq \delta^{-1} \operatorname{Vol}\left(\{\ell_t=0\}\cap B_r(x_0)\right).
    \end{equation}
    Hence, if $m_2-m_1<0$ we conclude that
    \begin{equation}
        \int_P f_t^{-n-1} \langle H_t,\Hess f_t\rangle dx \leq \frac{(m_2-m_1)\delta^{-1}}{\left(\frac{m_1+m_2}{2}\langle x_0,v_t\rangle+c_t\right)^{n+1}} \operatorname{Vol}(\{\ell_t=0\}\cap B_r(x_0)).
    \end{equation}
    To conclude, one just has to note that $\langle x,v_t\rangle$ is uniformly bounded in $t$, for $x\in P$: indeed,
    \(
        \abs*{\langle x,v_t\rangle}^2 \leq \norm{x}^2 
    \)
    since $v_t$ has unit norm. As $P$ is bounded, we can choose $c_t$ independently of $t>t_\delta$ in such a way that $f_t>1$ on $P$, for example. The resulting family of functions is $t$-uniformly bounded above on $P$ (for fixed $m_1,m_2$).

    \smallskip

    To show \eqref{eq:vHv}, let $A$ be any positive-definite symmetric matrix and let $w$ be any vector. By Cauchy-Schwarz,
    \begin{equation}
        \norm{w}^2 = \langle A^{1/2} w, A^{-1/2} w \rangle \leq \norm{A^{1/2} w} \norm{A^{-1/2} w} = (w^\transpose A w)^{1/2} (w^\transpose A^{-1} w)^{1/2},
    \end{equation}
    and in particular if $w$ has unit norm, $ w^\transpose A^{-1} w \geq (w^\transpose A w)^{-1} $. Applying this expression to $w=v_t$, $A=G_t$, we find, in $B_r(x_0)$,
    \begin{equation}
        v_t^\transpose H_t v_t \geq \left(v_t^\transpose G_t v_t\right)^{-1} > \delta^{-1}.  \qedhere
    \end{equation}
\end{proof}

\begin{corollary}
    Suppose that $\varphi_t$ is a smooth path of normalised symplectic potentials for $t<T$, and that there exist a vector $v_T$ of unit norm such that $\langle v_T, G_T(x)v_T \rangle$ vanishes to order $n-1$ near $x_0$ along $\{\ell_T=0\}$, i.e.\ there there exists $C>0$ such that
    \begin{equation}
        \langle v_T,G_T(x)v_T\rangle < C\abs{x-x_0}^{\alpha},\quad \alpha > n-1
    \end{equation}
    for every $x\in \{\ell_T=0\}\cap B_R(x_0)\subset P$. Then, with the same notation of Lemma~\ref{lemma:Gv_small},
    \begin{equation}
        \int_P f_t^{-n-1}\langle H_t, \Hess f_t\rangle dx \to -\infty\text{ as }t\to T.
    \end{equation}
\end{corollary}
\begin{proof}
    As $\langle v_T,G_T(x_0)v_T\rangle=0$, for every $\delta>0$ there exists $t_\delta<T$ and $r_\delta >0$ such that
    \begin{equation}
        0< \langle v_t,G_t(x)v_t\rangle <\delta
    \end{equation}
    for every $t$ between $t_\delta$ and $T$, and every $x\in \{\ell_t=0\}\cap B_{r_\delta}(x_0)$. By Lemma~\ref{lemma:Gv_small} then it follows that
    \begin{equation}\label{eq:accessoryineq}
        \int_P f_t^{-n-1}\langle H_t, \Hess f_t\rangle dx \leq 
        (m_2-m_1)k\,\delta^{-1} r_\delta^{n-1},
    \end{equation}
    where $k$ is a constant, not depending on $\delta$ nor on $t$. Under our assumptions, we can simply choose $r_\delta\sim \delta^{1/\alpha}$. In that case, \eqref{eq:accessoryineq} gives us
    \begin{equation}
        \int_P f_t^{-n-1}\langle H_t, \Hess f_t\rangle dx \sim  
        (m_2-m_1)k\,\delta^{\frac{n-1}{\alpha}-1}
    \end{equation}
    which goes to $-\infty$ for $\delta\to 0$.
\end{proof}

\noindent \textbf{Proof of Theorem~\ref{thm:boundary_behaviour}, case 1.} \quad We now turn to the behaviour of the CR Yamabe functional along paths $\varphi_\infty$ that go towards the ``boundary at infinity'' in the space of normalised symplectic potentials. More precisely, let $\varphi_\infty\not\equiv 0$ be a smooth, convex function, normalised so that $\varphi_\infty(o)=d\varphi_\infty(o)=0$. We consider a path $\varphi_t\in\mathcal{S}_o(P)$ such that
\begin{equation}\label{eq:asympt_path}
    \varphi_t = t^k\varphi_\infty+\text{ lower order terms}
\end{equation}
for some $k\in\bbn$. As $\Hess\varphi_\infty\geq 0$, these paths stay in the space of normalised symplectic potentials for arbitrarily large values of $t$.

\begin{proposition}\label{prop:infinitypath}
    For any path $\varphi_t\in\mathcal{S}_o(P)$ as in \eqref{eq:asympt_path} such that $\varphi_\infty$ is strictly convex,
    \begin{equation}
        \lim_{t\to +\infty}\CRYam(\varphi_t)\leq 0.
    \end{equation}
\end{proposition}
\begin{proof}
    We consider the following class of test functions
    \begin{equation}
        f_{m,\varepsilon}(x) \coloneqq \left(\varepsilon + \alpha d(x,o)^2\right)^m
    \end{equation}
    where $\varepsilon>0$, $m$ is a positive integer, and $\alpha>0$ is chosen (depending on $o$) so that $f_{m,\varepsilon}>1$ on $\partial P$. Then, we can estimate the various pieces entering in the definition of the CR Einstein--Hilbert functional as follows. First, as our test functions are greater than $1$ on the boundary,
    \begin{equation}
        \int_{\partial P} f_{m,\varepsilon}^{-n}d\sigma < \int_{\partial P} d\sigma.
    \end{equation}
    Moreover, we have $f_{m,\varepsilon}(x)<(2\varepsilon)^m$ for every $x$ in a ball of radius $\sqrt{\varepsilon/\alpha}$ centered at $o$ (that we can suppose to be contained in the interior of $P$). Hence,
    \begin{equation}
        \int_P f_{m,\varepsilon}^{-n-1}d\mu > \int_{B_{\sqrt{\varepsilon/\alpha}}(o)} f_{m,\varepsilon}^{-n-1}d\mu \geq (2\varepsilon)^{-m(n+1)}\int_{B_{\sqrt{\varepsilon/\alpha}}(o)}d\mu
    \end{equation}
    so that
    \begin{equation}
        \left(\int_P f_{m,\varepsilon}^{-n-1}d\mu\right)^{-\frac{n}{n+1}} \leq (2\varepsilon)^{mn}\left(\int_{B_{\sqrt{\varepsilon/\alpha}}(o)}d\mu\right)^{-\frac{n}{n+1}}.
    \end{equation}
    Note that we can make this quantity infinitesimal as $\varepsilon\to 0$, by choosing $m$ appropriately.

    To estimate the Hessian term, let's first compute the second derivative of $f_{m,\varepsilon}$. The direct computation gives
    \begin{equation}
        \partial_i\partial_j f_{m,\varepsilon} = 4m(m-1)\alpha^2 f_{m,\varepsilon}\frac{(x_i-o_i)(x_j-o_j)}{\left(\varepsilon+\alpha d(x,o)^2\right)^2} + 2m\alpha f_{m,\varepsilon}\frac{\delta_{ij}}{\varepsilon+\alpha d(x,o)^2},
    \end{equation}
    which when paired with some other matrix $H$ gives
    \begin{equation}
        \langle H, \Hess f_{m,\varepsilon} \rangle = 4m(m-1)\alpha^2 f_{m,\varepsilon}\frac{\langle (x-o), H(x-o)\rangle}{\left(\varepsilon+\alpha d(x,o)^2\right)^2} + 2m\alpha f_{m,\varepsilon}\frac{\Tr(H)}{\varepsilon+\alpha d(x,o)^2}.
    \end{equation}
    Now, when taking the path $\varphi_t = t^k \varphi_\infty+O(t^{k-1})$, we can expand in $t$ as $t\to+\infty$ to find, since $\Hess\varphi_\infty$ is invertible,
    \begin{equation}
        H_t = t^{-k}\left((\Hess\varphi_\infty)^{-1}+O(t^{-1})\right).
    \end{equation}
    Now we can expand in $t$ the whole integral, assuming for the moment that $\varepsilon$ does not depend on $t$.
    \begin{equation}
    \begin{split}
        \int_P f_{m,\varepsilon}^{-n-1} & \langle H_t, \Hess f_{m,\varepsilon} \rangle d\mu =\\
        =& 4m(m-1)\alpha^2 \int_P f_{m,\varepsilon}^{-n} \frac{\langle (x-o), H_t(x-o)\rangle}{\left(\varepsilon+\alpha d(x,o)^2\right)^2} d\mu + 2m\alpha \int_P f_{m,\varepsilon}^{-n} \frac{\Tr(H)}{\varepsilon+\alpha d(x,o)^2} d\mu\\
        =& \frac{4m(m-1)\alpha^2}{t^k} \int_P f_{m,\varepsilon}^{-n} \frac{\langle (x-o), \Hess(\varphi_\infty)^{-1}(x-o)\rangle}{\left(\varepsilon+\alpha d(x,o)^2\right)^2} d\mu\\
        &+ \frac{2m\alpha}{t^k} \int_P f_{m,\varepsilon}^{-n} \frac{\Tr(\Hess(\varphi_\infty)^{-1})}{\varepsilon+\alpha d(x,o)^2} d\mu + O(t^{-k-1}).
    \end{split}
    \end{equation}
    Fix some $\delta>0$; we can choose $o\in P^\circ$ and $\alpha,\varepsilon,m$ so that
    \begin{enumerate}
        \item $f_{m,\varepsilon}>1$ on $\partial P$;
        \item $(2\varepsilon)^{m}\operatorname{Vol}(B_{\sqrt{\varepsilon/\alpha}}(0))^{-\frac{1}{n+1}} < \left(\delta \operatorname{Vol}(\partial P,d\sigma)^{-1}\right)^{\frac{1}{n}}$.
    \end{enumerate}
    Then, we find
    \begin{equation}
    \begin{split}
        \EH(f_{m,\varepsilon},H_t) =& \frac{\int_{\partial P} f_{m,\varepsilon}^{-n}d\sigma}{\left(\int_P f_{m,\varepsilon}^{-n-1}d\mu\right)^{\frac{n}{n+1}}} + \frac{\int_P f_{m,\varepsilon}^{-n-1} \langle H_t, \Hess f_{m,\varepsilon} \rangle d\mu}{\left(\int_P f_{m,\varepsilon}^{-n-1}d\mu\right)^{\frac{n}{n+1}}}\\
        < &   \left(\int_{\partial P}d\sigma\right)  (2\varepsilon)^{mn}\left(\int_{B_{\sqrt{\varepsilon/\alpha}}(o)}d\mu\right)^{-\frac{n}{n+1}}\\
        &+ (2\varepsilon)^{mn}\left(\int_{B_{\sqrt{\varepsilon/\alpha}}(o)}d\mu\right)^{-\frac{n}{n+1}}
        \frac{4m(m-1)\alpha^2}{t^k} \int_P f_{m,\varepsilon}^{-n} \frac{\langle (x-o), \Hess(\varphi_\infty)^{-1}(x-o)\rangle}{\left(\varepsilon+\alpha d(x,o)^2\right)^2} d\mu\\
        &+ (2\varepsilon)^{mn}\left(\int_{B_{\sqrt{\varepsilon/\alpha}}(o)}d\mu\right)^{-\frac{n}{n+1}}\frac{2m\alpha}{t^k} \int_P f_{m,\varepsilon}^{-n} \frac{\Tr(\Hess(\varphi_\infty)^{-1})}{\varepsilon+\alpha d(x,o)^2} d\mu + O(t^{-k-1})\\
        <& \delta + \delta\frac{4m(m-1)\alpha^2}{t^k\operatorname{Vol}(\partial P,d\sigma)} \int_P f_{m,\varepsilon}^{-n} \frac{\langle (x-o), \Hess(\varphi_\infty)^{-1}(x-o)\rangle}{\left(\varepsilon+\alpha d(x,o)^2\right)^2} d\mu\\
        &+\delta\frac{2m\alpha}{t^k\operatorname{Vol}(\partial P,d\sigma)} \int_P f_{m,\varepsilon}^{-n} \frac{\Tr(\Hess(\varphi_\infty)^{-1})}{\varepsilon+\alpha d(x,o)^2} d\mu + O(t^{-k-1}).
    \end{split}
    \end{equation}
    The terms in these last two integrals (and in the $O(t^{-k-1})$ part) are not $+\infty$; so we can take $t$ to be extremely large (depending on $\delta$, as well as the other parameters) to obtain \( \EH(f_{m,\varepsilon},H_t) < 4 \delta. \)
    Hence, for every $\delta$ there exists $T$ such that
    \(
        \CRYam(H_t) < 4\delta
    \)
    for every $t>T$. 
\end{proof}

Roughly, the proof of Proposition~\ref{prop:infinitypath} shows that for a positive function $f$ and a strictly convex, smooth function $\varphi_\infty$,
\begin{equation}
    \lim_{t\to+\infty}\EH(f,\varphi_0+t\varphi_\infty) = \InfFun(f),
\end{equation}
and in particular the limit does not depend on the symplectic potential. Here, we define
\begin{equation}\label{eq:InfFun_def}
    \InfFun(f) \coloneqq \frac{2\int_{\partial P} f^{-n} d\sigma}{\left(\int_P f^{-n-1}dx\right)^{n/n+1}}.
\end{equation}
Of course, the infimum of ${\bf I}$ over all positive smooth functions on $P$ is zero.

Notice that when $f$ is affine-linear, then $\InfFun(f)=\EH(\phi,f)$ for every $\varphi\in\mathcal{S}(P)$. Moreover, $\InfFun$ is defined for conformal factors $f$ that are not necessarily smooth. We will show in the next section that in fact, $\InfFun(f_{pl})$ for a concave, piecewise-linear function $f_{pl}$, is strictly related to the \emph{algebraic CR Einstein--Hilbert functional} introduced in \cite{LahdiliLegendreScarpa_EHDF}.

\section{The Action functional of a toric Sasaki manifold}\label{sec:action}

One of the main results of \cite{LahdiliLegendreScarpa_EHDF,LahdiliLegendreScarpa_CRYam} is that the CR Einstein--Hilbert functional detects the K-stability of Sasaki manifolds. In this Section, we briefly recall the main points of the construction and then show how to recover some results of \cite{LahdiliLegendreScarpa_EHDF,LahdiliLegendreScarpa_CRYam} more directly. In turn, this leads to formulate Question~\ref{q:CRYam_testconf}, relating the toric CR Yamabe invariant of $(N,D)$ to stability properties of the polytope $P_{\xi_0}$. 

\smallskip

For a test configuration $(\tstX,\tstL)$ for a polarised K\"ahler manifold $(X,L)$ and a Reeb field $\chi\in\Reebcone$, let $\omega_t$ be the geodesic ray of K\"ahler metrics associated to $(\tstX,\tstL)$, and let $\eta_t$ be a corresponding path of Sasaki forms on the $U(1)$-bundle associated to $L\to X$. Then, there is a $2$-parameter path (a ``ribbon'') of conformal factors $f_{t,s}$, such that for every small enough $s$, the limit
\begin{equation}
    \EH^\chi_s(\tstX,\tstL) \coloneqq \lim_{t\to+\infty}\EH(f_{s,t}^{-1}\eta_t)
\end{equation}
exists and does not depend on the starting point of the geodesic ray. Moreover, there is an asymptotic expansion
\begin{equation}
    \EH^\chi_s(\tstX,\tstL) = \EH(\chi) + s\operatorname{DF}(\tstX,\tstL) + O(s^2),
\end{equation}
where $\operatorname{DF}(\tstX,\tstL)$ is the Donaldson-Futaki weight of the test configuration.

\smallskip

As in the general setting, to study the CR Einstein--Hilbert functional along a \emph{toric} geodesic ray, we need to evaluate $\EH$ on symplectic potentials $\varphi$ that are not smooth, but rather just continuous and whose second variation is a positive current. Indeed, a general geodesic ray in the toric setting is a line of the form
\begin{equation}
    \varphi_t = \varphi_0 + t \varphi_{pl}
\end{equation}
for a convex, piecewise-linear function $\varphi_{pl}$ on the polytope. While one can make sense of the Hessian of such a function, in the sense of distributions, algebraic manipulations such as taking the inverse Hessian
\begin{equation}
    \varphi\mapsto \operatorname{Hess}(\varphi)^{-1}
\end{equation}
are not well-defined, so it is not clear how one might extend the CR Einstein--Hilbert functional to this setting.

This difficulty can be overcome by considering the \emph{action} along the ribbon $f_{s,t}^{-1}\eta_t$. In the present setting, this corresponds to considering
\begin{equation}
    \mathcal{A}_s(T) \coloneqq \int_0^\Torus\Scaltot(f_{s,t},\varphi_t)dt
\end{equation}
for a \emph{smooth} path of symplectic potentials $\varphi_t$ (and a corresponding ribbon of conformal factors $f_{s,t}$), and then showing that the resulting function of one real variable $t\mapsto\mathcal{A}_s(t)$ is actually well-defined even if $f_{s,t}$ and $\varphi_t$ have less regularity. This gives a way of extending the definition of $\EH$ by defining $\Scaltot(f_{s,t},u_t)\coloneqq\partial_t\mathcal{A}_s(t)$ for less regular data.

\smallskip

Hence, let us consider a smooth geodesic ray $\varphi_t \coloneqq  \varphi_0 + t \varphi_\infty$. For any $\chi\in\Reebcone$, the ribbon corresponding to $\varphi_\infty$ and $\chi$ is (see \cite{LahdiliLegendreScarpa_EHDF})
\begin{equation}
    f_{s,t} \coloneqq \ell_\chi-s\varphi_\infty,
\end{equation}
where $\ell_\chi$ is an affine-linear function on $P$. As in this case $f_{s,t}$ does not depend on $t$, we will simply denote it by $f_s$.

The total scalar curvature along the ribbon is 
\begin{equation}
    \Scaltot(f_{s},\varphi_t)=2\int_{\partial P}f_{s}^{-n}d\sigma + n\int_P f_{s}^{-n-1}\operatorname{Tr}(\operatorname{Hess}(\varphi_t)^{-1}\operatorname{Hess}(f_{s}))dx.
\end{equation}
Note however that
\begin{equation}\label{eq:hessian_condition}
    \operatorname{Hess}(f_{s})=-s\operatorname{Hess}(u_\infty)=-s\partial_t\operatorname{Hess}(\varphi_t).
\end{equation}
Substituting this into the expression of $\Scaltot$ then we find
\begin{equation}
\begin{split}
    \Scaltot(f_{s},\varphi_t) = & 2\int_{\partial P}f_{s}^{-n}d\sigma -ns\int_P f_{s}^{-n-1}\operatorname{Tr}(\operatorname{Hess}(\varphi_t)^{-1}\partial_t\operatorname{Hess}(\varphi_t))dx\\
    = & \partial_t\left[2t\int_{\partial P}f_{s}^{-n}d\sigma -ns\int_P f_{s}^{-n-1}\log\det\left(\operatorname{Hess}(\varphi_t)\right)dx\right].
\end{split}
\end{equation}
Hence, we define the \emph{action functional} along the ribbon $(f_{s},\varphi_t)$ as
\begin{equation}
    \mathcal{A}^\chi_s(t) \coloneqq 2t\int_{\partial P}f_{s}^{-n}d\sigma -ns\int_P f_{s}^{-n-1}\log\det\operatorname{Hess}(\varphi_t)\,dx.
\end{equation}
This expression however might be problematic if $\varphi_\infty$ is piecewise-linear: in that case, the entries of $\Hess(\varphi_t)$ include some delta functions of which we might not be allowed to take products (necessary to compute the determinant) or the logarithms.

To circumvent the issue, note that for a smooth convex function $\varphi$, $\det\Hess\varphi$ is the density of the Riemannian measure induced by $\varphi$ (the Monge-Ampère measure of $\varphi$). Formally then we can interpret $\log\det\operatorname{Hess}(\varphi)\,dx$ as
\begin{equation}
    \log\left(\frac{d\operatorname{MA}(u_t)}{dx}\right)dx,
\end{equation}
i.e.\ as a relative entropy of two measures on $P$. Note that
\begin{equation}
    \operatorname{MA}(\varphi_t)=\operatorname{MA}(\varphi_0)+t\operatorname{MA}(\varphi_\infty)
\end{equation}
is the Lebesgue decomposition of the Monge-Ampère measure with respect to the Lebesgue measure $dx$: the absolutely continuous part is $\operatorname{MA}(\varphi_0)$, while $\operatorname{MA}(\varphi_\infty)\perp dx$ (cf.\ \cite[Theorem 6.10]{Rudin}). Then, we can appeal to \cite[Theorem 7.14]{Rudin} to compute the Radon-Nikodym derivative of $\operatorname{MA}(\varphi_t)$ with respect to $\mu$, since the singular part has zero derivative:
\begin{equation}
    \frac{d\operatorname{MA}(\varphi_t)}{dx} = \frac{d\operatorname{MA}(\varphi_0)}{dx}.
\end{equation}
Hence, the action functional is
\begin{equation}\label{eq:act_func_geodribbon}
    \mathcal{A}^\chi_s(t) \coloneqq 2t\int_{\partial P}f_{s}^{-n}d\sigma -ns\int_P f_{s}^{-n-1}\log\det\operatorname{Hess}(\varphi_0)\,dx.
\end{equation}
As this expression is affine in $t$, we obtain, for the limit of the CR Einstein--Hilbert functional,
\begin{proposition}
    For any $s>0$ such that $f_s=\ell_\chi-s\varphi_\infty>0$,
    \begin{equation}
        \EH_s^\chi(\mathcal{X},\mathcal{L}) = \lim_{t\to +\infty} \frac{\partial_t\mathcal{A}^\chi_s(t)}{\left(\int_P f_s^{-n-1}dx\right)^{\frac{n}{n+1}}} = \InfFun(f_s),
    \end{equation}
    where $\InfFun(f)$ is defined in \eqref{eq:InfFun_def}, and $(\tstX,\tstL)$ is the test configuration associated to the piecewise-linear function $\varphi_\infty$.
\end{proposition}
Note that $f_s=\ell_\chi-s\varphi_\infty$ is a \emph{concave} piecewise linear function, for $s>0$. Hence, we obtain a new proof in our setting of \cite[Proposition $5.7$, Corollary $5.8$]{LahdiliLegendreScarpa_CRYam}.
\begin{corollary}
    For any $\varphi_0\in\mathcal{S}(P)$ and $0<s\ll 1$,
    \begin{equation}\label{eq:EH_ineq}
        \EH(\ell_\chi-s\varphi_\infty,\varphi_0)<\InfFun(\ell_\chi-s\varphi_\infty).
    \end{equation}
\end{corollary}
We can also reformulate a question of \cite{LahdiliLegendreScarpa_CRYam} as follows:
\begin{question}\label{q:CRYam_testconf}
    Let $P$ be a Delzant polytope. Is it true that
   \begin{equation}
        \sup_{\varphi\in\mathcal{S}(P)}\CRYam(\varphi) = \inf_{f>0,\text{ concave}} \InfFun(f)\ ?
    \end{equation}
\end{question}
The inequality ``$\leq$'' follows from \eqref{eq:EH_ineq}. Note that, while $\inf_{f>0}\InfFun(f)=0$, the infimum of $\InfFun$ over \emph{concave} functions is related to the K-stability of the polytope $P$, as any concave function on $P$ can be approximated by piecewise-linear ones, which are associated to test configurations as above.

\begin{remark}
    Equation \eqref{eq:act_func_geodribbon} shows that the action functional is \emph{affine} along a geodesic ribbon, so we can pose, \emph{for every $t\geq 0$} and $s>0$ sufficiently small, 
    \begin{equation}
        \EH(\ell_\chi - s \varphi_\infty , \varphi_0 + t \varphi_\infty) = \frac{\left(\partial_t\mathcal{A}^\chi_s(t)\right)_{t=0}}{\left(\int_P (\ell_\chi - s \varphi_\infty)^{-n-1}dx\right)^{\frac{n}{n+1}}} = \InfFun(\ell_\chi-s\varphi_\infty).
    \end{equation}
\end{remark}
The key step in this computation is the relation \eqref{eq:hessian_condition} between the Hessian of the conformal factor and the symplectic potential. Hence, we can extend $\EH$ as follows:
\begin{equation}
    \EH(\ell+a\varphi,\varphi_0+\varphi) \coloneqq \InfFun(\ell+a\varphi)
\end{equation}
for every $\varphi_0\in\mathcal{S}(P)$ and any $\varphi$, affine-linear $\ell$, and constant $a$ such that
\begin{enumerate}
    \item $\ell+a\varphi\in\mathcal{C}^0(\bar{P},\mathbb{R}_{>0})$;
    \item $\Hess(\varphi_0+\varphi)$ is a positive current with $L^\infty(P)$ coefficients.
\end{enumerate}

\section{Examples}\label{sec:examples}
To enrich the theory in this paper, we finish with this section supporting examples. In \S\ref{sec:CP1} we continue our investigation of the toric $3$-sphere over $\bbc\bbp^1$ started in \S\ref{ssCRyam2sphere}. In particular we give a very explicit case of  $\CRYam^\bbt(\varphi_a)$ being negative for a specific deformation $\varphi_a$ of the standard Fubini-Study potential, where the path $\varphi_a$ is of different nature from the ones considered in Theorem~\ref{thm:boundary_behaviour}. Instead of working directly with a deformation of a symplectic potential, in Example~\ref{ex:cp1negative} we deform the inverse of the Hessian of the potential. In \S\ref{sec:proj_bundle} (specifically \S\ref{ex:negativeHirzebruchsurf}) we construct a similar example on the non-trivial Hirzebruch surfaces, using the standard moment map (generalized Calabi/admissible) construction.

The subsection \S\ref{sec:trivialHirzebruch} gives a careful review and treatment of the case of Sasaki structures over the trivial Hirzebruch surface, $\bbc\bbp^1\times\bbc\bbp^1$. Building on results from \cite{Legendre_toricSasaki} and \S 2.4 of \cite{LahdiliLegendreScarpa_CRYam}, we are able to bring some illumination to 
Question~\ref{q:uniqueness} and also construct some new explicit non-Sasaki cscTW examples.
These cscTW solutions are not CR Yamabe minimizers.

Finally, in \S\ref{sec:nontoric} we will venture out of the toric setting in order to investigate cases where the Sasaki cone is only $2$-dimensional and yet has enough nuance with regard to the Einstein-Hilbert functional that the examples serve as further food for thought regarding Question~\ref{q:uniqueness}.

\subsection{\texorpdfstring{$S^1$}{U(1)} symmetric K\"ahler metrics on \texorpdfstring{$\bbc\bbp^1$}{CP1}}\label{sec:CP1}

In the usual toric set-up for $\bbc\bbp^1$ we consider K\"ahler metrics of the form 
\begin{equation}\label{cp1toric}
g_H=\frac {d\gz^2}
{H(\gz)}+H(\gz)d\phi^2,\quad
\omega_{\bbc\bbp^1} = d\gz \wedge d\phi,
\end{equation}
where $-1<\gz<1$, $0<\phi\leq 2\pi$, and
\begin{align}\label{endpointscp1}
(i)H(\gz) > 0, \quad -1 < \gz <1,\quad
(ii) H(\pm 1) = 0,\quad
(iii) H'(\pm 1) = \mp 2.
\end{align}
Let now $(g_H,\omega_{\bbc\bbp^1},J_H)$ be defined by some choice of $H:[-1,1]\rightarrow \bbr$ satisfying~\eqref{endpointscp1}. 

The scalar curvature of $g_H$ is given by 
$Scal=-H''(\gz)$ and if $u=u(\gz)$ is a smooth function of $\gz$ (and hence $u$ can be viewed as a smooth function on $\bbc\bbp^1$) and if $\Delta_{g_H}$ denotes the Laplacian of 
$g_H$, then $\Delta_{g_H} u =-\frac{d}{d\gz}\left[H(\gz)u'(\gz)\right]$.

We know that $\int_{\bbc\bbp^1}\omega_{\bbc\bbp^1}=4\pi$, so $\left[\frac{\omega_{\bbc\bbp^1}}{2\pi}\right]$ is $2$ times a generator of $H^2(\bbc\bbp^1,\bbz)$. After a rescale of $1/2$ (which we will ignore) we can consider $S^3$ as the Boothby-Wang $S^1$-bundle over $(\bbc\bbp^1, g_H,\omega_{\bbc\bbp^1})$. Let $\eta=\eta_{\omega_{\bbc\bbp^1}}$ be a corresponding contact $1$-form on $S^3$. 
Then the period of the Reeb orbit is $\frac{2\pi}{2}=\pi$ and if $L_\xi$ is a generic leaf of the foliation $\calf_\xi$, then $\int_{L_\xi}\eta = \pi$.

By construction $d\eta=\pi^*\omega_{\bbc\bbp^1}$ and
by Frobenius' formula for iterated integrals, we then have that,
if $u=u(\xi)$ is a smooth function of $\xi$ (and viewed via a pull-back as a function on $S^3$), then
$$
\int_M u\eta\wedge d\eta
= \int_0^{2\pi}\int_{-1}^1\left(\int_{L_\xi}\eta\right)u d\gz\wedge d\phi
= 2\pi^2 \int_{-1}^1 u d\gz.
$$

If $f:\bbc\bbp^1\rightarrow \bbr^+$ is any smooth positive function we know that
$$
\begin{array}{ccl}
Scal^{TW}(H,f)& = &f\, Scal_{g_H}-4\Delta_{g_H}f -6f^{-1}|df|^2_{g_H}\\
\\
&=& -f\,H''(\gz)+4\frac{d}{d\gz}\left(H(\gz)f'(\gz)\right)-6\frac{(f'(\gz))^2}{f(\gz)}H(\gz).
\end{array}
$$
Thus, the total scalar curvature is
$$
\begin{array}{ccl}
\mathbf{Scal}(H,f)&=& \int_{S^3}Scal^{TW}(f^{-1}\eta_{\omega_{\bbc\bbp^1}})f^{-2}\eta\wedge d\eta\\
\\
&=& 2\pi^2 \int_{-1}^1\left(-f(\gz)^{-1}\,H''(\gz)+4f(\gz)^{-2}\frac{d}{d\gz}\left(H(\gz)f'(\gz)\right)-6\frac{(f'(\gz))^2}{f(\gz)^3}H(\gz)\right)d\gz\\
\\
&=& 2\pi^2 \int_{-1}^1 f(\gz)^{-1}\left(-H''(\gz)+2\left(\frac{f'(\gz)}{f(\gz)}\right)^2 H(\gz)\right)d\gz,
\end{array}
$$
where the last equality follows from integration by parts and \eqref{endpointscp1}.
We also have that
$$\int_{S^3}f^{-2}\eta\wedge d\eta=2\pi^2 \int_{-1}^1 f(\gz)^{-2} d\gz.$$
Therefore, up to an overall positive rescale that does not depend on $f$, we have that
$$EH(H,f)= \frac{\int_{-1}^1 f(\gz)^{-1}\left(-H''(\gz)+2\left(\frac{f'(\gz)}{f(\gz)}\right)^2 H(\gz)\right)d\gz}{\left(\int_{-1}^1 f(\gz)^{-2} d\gz\right)^\frac{1}{2}}.$$
We see that if $H''(\gz)<0$, as for example when $H(\gz)=1-\gz^2$ (corresponding to the Fubini-Study metric), then $EH(H,f)$ is positive. However, as the example shows below, it is possible to chose some $H$ and $f$ such that $EH(H,f)$ becomes negative.

\begin{example}\label{ex:cp1negative}
Now consider the following deformation of $H(\gz)=1-\gz^2$:
$$H_a(\gz)=1-\gz^2 +a (1 - \gz^2)^2, $$
where $a> -1$. Consider $\hat{f}(\gz)=2-\gz^2$. This is a positive smooth function for $-1<\gz<1$ and it is not a killing potential. We calculate that up the the same overall scale as above, the CR Einstein--Hilbert functional is
$$
\begin{array}{cl}
&EH(H_a,\hat{f})\\
\\
=& \frac{10-3 \sqrt{2} \tanh ^{-1}\left(\frac{1}{\sqrt{2}}\right)}{\sqrt{2+\sqrt{2} \tanh ^{-1}\left(\frac{1}{\sqrt{2}}\right)}}\\
\\
+& \left(\frac{7 \sqrt{2} \tanh ^{-1}\left(\frac{1}{\sqrt{2}}\right)-10}{\sqrt{2+\sqrt{2} \tanh ^{-1}\left(\frac{1}{\sqrt{2}}\right)}}\right)a.
\end{array}
$$
Since $\left(\frac{7 \sqrt{2} \tanh ^{-1}\left(\frac{1}{\sqrt{2}}\right)-10}{\sqrt{2+\sqrt{2} \tanh ^{-1}\left(\frac{1}{\sqrt{2}}\right)}}\right)\approx -0.707544$ is negative, this will be negative for sufficiently large $a$ values.
Indeed, setting $\hat{a}=\frac{3 \sqrt{2} \tanh ^{-1}\left(\frac{1}{\sqrt{2}}\right)-10}{7 \sqrt{2} \tanh ^{-1}\left(\frac{1}{\sqrt{2}}\right)-10}\approx 4.9109$, we have that $EH(H_a,\hat{f})>0$ for $a<\hat{a}$, $EH(H_a,\hat{f})=0$ for $a=\hat{a}$, and $EH(H_a,\hat{f})<0$ for $a>\hat{a}$.
\end{example}
\subsection{The Trivial Hirzebruch Surface}\label{sec:trivialHirzebruch}
This section starts with a review of known material 
of the Sasaki structures over the trivial Hirzebruch surface, $\bbc\bbp^1\times\bbc\bbp^1$,
from \cite{Legendre_toricSasaki} and Section 2.4 of \cite{LahdiliLegendreScarpa_CRYam}. Subsequently we will see that this case is a nice toy example for Question~\ref{q:uniqueness} and further we will be able to construct some explicit non-Sasaki cscTW examples.
These cscTW solutions are not CR Yamabe minimizers.

Given $p,q\in \bbr^+$, consider the rectangle $P_{p,q}$ in the $xy$-plane given by  $-p<x<p$ and $-q<y<q$.
Then a  product K\"ahler structure $(g,\omega)$ on $\bbc\bbp^1\times\bbc\bbp^1$ can be described on a open dense set as follows:
\begin{equation}\label{productmetric}
g= \frac{1}{A(x)} dx^2+A(x) dt^2 + \frac{1}{B(y)} dy^2 + B(y) ds^2, \quad \omega = dx\wedge dt + dy\wedge ds,
\end{equation}
where $p,q >0$, $-p<x<p$, $-q <y<q$, $t,s$ are angular coordinates, and $A(x)$, $B(y)$ are positive smooth functions. 
For $g$ to extend to all of $\bbc\bbp^1\times\bbc\bbp^1$ as a smooth K\"ahler metric we
have the following boundary conditions:
\begin{align}
\begin{split}
A(\pm p)=0 \; \; & \;\; B(\pm q)=0\\
A'(\pm p)=\mp 2\; \; & \;\; B'(\pm q)=\mp 2
\end{split}
\end{align}
The K\"ahler class of such a K\"ahler structure is given by $2\pi\left(p c_1(\calo(2))+ qc_1(\calo(2))\right)$ and by
varying the values of $p,q \in \bbr^+$ we exhaust the  $2$-dimensional K\"ahler cone of $\bbc\bbp^1\times\bbc\bbp^1$ with 
product K\"ahler metrics whose moment polytope is exactly $P_{p,q}$. The scalar curvature of $g$ is given by
$Scal(g)=-A''(x)-B''(y)$.

We now assume that $p,q\in \bbz^+$ and $\gcd(p,q)=1$. Then $[\omega]$ is - up to an overall multiple of $4\pi$ - a primitive integer K\"ahler class on
$\bbc\bbp^1\times\bbc\bbp^1$. 

The corresponding Boothby-Wang constructed Sasaki structure $(N,\cald, J, \eta)$ above $\bbc\bbp^1\times\bbc\bbp^1$
has (unreduced) Sasaki-Reeb cone $\kt_+$  identified with the space of positive affine-linear functions over $P_{p,q}$.
Up to scale, we may assume that these functions (whose pull-back to $\bbc\bbp^1\times\bbc\bbp^1$ are killing potentials)
are of the form $f(x,y)=p^2+q^2+c_A x+c_B y$ and thus, up to scale $\kt_+$
is identified with the quadrilateral 
$$D_{p,q}=\{(c_A,c_B)\in {\bbr}^2\,|\,p^2+q^2\pm c_Ap\pm c_Bq>0\,\wedge\,p^2+q^2\pm c_Ap\mp c_Bq>0\}.$$
Put in another way, $D_{p,q}$ parametrizes the rays in $\kt_+$.

Any smooth function $f:P_{p,q} \rightarrow \bbr^+$ given by $f(x,y)$ may be viewed as a $\bbt^2$-invariant smooth positive function 
on $\bbc\bbp^1\times\bbc\bbp^1$. As such, we can consider $\Delta_g(f)$ and $|df|^2_g$:
\begin{equation}\label{laplacianANDother}
\begin{array}{ccl}
\Delta_g(f) &=& -\frac{\partial}{\partial x}\left(A\frac{\partial f}{\partial x}\right)-\frac{\partial}{\partial y}\left(B\frac{\partial f}{\partial y}\right)\\
\\
|df|^2_g &=& \left(\frac{\partial f}{\partial x}\right)^2 A+ \left(\frac{\partial f}{\partial y}\right)^2 B
\end{array}
\end{equation}

We now have that 
\begin{equation}\label{scaltwtrivialhirzebruch}
\begin{array}{ccl}
Scal^{TW}(A,B,f) & = & f Scal (g)-6\Delta_g f-12|df|^2_g f^{-1}\\
\\
&=& -f\left(A''(x)+B''(y)\right)+6\left(\frac{\partial}{\partial x}\left(A(x)\frac{\partial f}{\partial x}\right)+\frac{\partial}{\partial y}\left(B(y)\frac{\partial f}{\partial y}\right)\right) \\
\\
&- & 12 \left( \left(\frac{\partial f}{\partial x}\right)^2 A(x)+ \left(\frac{\partial f}{\partial y}\right)^2 B(y)\right)f^{-1}
\end{array}
\end{equation}
and, up to an overall fixed scaling factor,
\begin{equation}\label{energytrivialhirzebruch}
EH(A,B,f) = \frac{\int_{-q}^{q}\int_{-p}^p f^{-3}Scal^{TW}(f^{-1}\eta) \,dx\,dy}{\left(\int_{-q}^{q}\int_{-p}^p f^{-3}\,dx\,dy\right)^{2/3}}.
\end{equation}

\subsubsection{Constant Scalar Curvature Sasaki metrics}

From \cite{Legendre_toricSasaki} we know that if we choose 
$$A(x)=\frac{p^2-x^2}{p}+\left(\frac{c_A^2 p (p+q)}{2 q \left(-c_A^2 p^2+3 p^4+6 p^2 q^2+3 q^4\right)}\right) \left(\frac{p^2-x^2}{p}\right)^2$$
and $B(y)=\frac{q^2-y^2}{q}$, then
for $f(x,y)=p^2+q^2+c_A x$, $Scal(g)_{f,4}=f Scal^{TW}(f^{-1}\eta)$ is a killing potential (linear affine function in $x$) and thus,
for all $(c_A,0)\in D_{p,q}$, we have an extremal Sasaki structure.
When $c_A=0$, of course this extremal Sasaki metric is actually cscS and its regular K\"ahler quotient is just the product cscK structure
on $\bbc\bbp^1\times\bbc\bbp^1$. When $p>5q$, as discovered by Legendre in \cite{Legendre_toricSasaki}, we also have cscS solutions when
$c_A=\pm\frac{\left(p^2+q^2\right) }{p}\sqrt{\frac{p-5 q}{p-q}}$ ($(c_A,0)$ is safely inside $D_{p,q}$) and then
\begin{equation}\label{cscSxtraAfunction}
A(x)=\frac{p^2-x^2}{p}+\left(\frac{p - 5 q}{4 p q}\right) \left(\frac{p^2-x^2}{p}\right)^2.
\end{equation}
Since $A(x)$ only depends on $|c_A|$ and $B(y)$ is already fixed, these extra cscS metrics are technically extremal Sasaki twins as defined in \cite{BHLT25}, albeit related by the Weyl group. See Section 3.6 in \cite{BHLT25} for more details.
\begin{remark}\label{symmetry}
By symmetry, similar results as above follows if we instead chose $A(x)=\frac{p^2-x^2}{p}$,
$B(y)=\frac{q^2-y^2}{q}+ \left(\frac{c_B^2 q (p+q)}{2 p \left(-c_B^2 q^2+3 q^4+6 p^2 q^2+3 p^4\right)}\right) \left(\frac{q^2-y^2}{q}\right)^2$, and
$f(x,y)=p^2+q^2+c_B y$.
\end{remark}

\subsubsection{CR Einstein--Hilbert energy as a function on $D_{p,q}$}\label{sec:EH_Dpq}

For convenience of calculations, we now let $A(x)=\frac{p^2-x^2}{p}$ and $B(y)=\frac{q^2-y^2}{q}$ and use
\eqref{energytrivialhirzebruch} to
calculate $EH(f^{-1}\eta_\omega)$ for $f(x,y)=p^2+q^2+c_A x+c_B y$. Denoting the resulting
$EH(A,B,f)$ by $EH(c_A,c_B)$, this gives us a function on $D_{p,q}$:

$$EH(c_A,c_B)=\frac{2\cdot 2^{2/3} \left(-c_A^2 p^2 (p-q)+c_B^2 q^2 (p-q)+(p+q) \left(p^2+q^2\right)^2\right)}{\sqrt[3]{p^2 q^2 \left(p^2+q^2\right)^2 \left(\left(p^2+q^2\right)^2-(c_A p-c_B q)^2\right)
 \left(\left(p^2+q^2\right)^2-(c_A p+c_B q)^2\right)}}$$
 
For $p\leq 5q$ and $q\leq 5p$, it is easy to check that $EH(c_A,c_B)$ has a global minimum of value $(\frac{32 (p+q)^3}{p^2 q^2})^{1/3}$ at $(c_A,c_B)=(0,0)$.
However, if $p>5q$, the point $(c_A,c_B)=(0,0)$ is the location of a saddle point and we have
two global minima of value $(\frac{432 (p-q)^2}{p^2 q})^{1/3}$ at $(c_A,c_B)=(\pm\frac{\left(p^2+q^2\right) }{p}\sqrt{\frac{p-5 q}{p-q}},0)$. Likewise, if
$q>5p$, the point $(c_A,c_B)=(0,0)$ is the location of a saddle point and we have
two global minima of value $(\frac{432 (p-q)^2}{q^2 p})^{1/3}$ at $(c_A,c_B)=(\pm\frac{\left(p^2+q^2\right) }{q}\sqrt{\frac{q-5 p}{q-p}},0)$.
See Figure~\ref{2min1saddle}.
No other critical points appear.

\begin{figure}[t]
\includegraphics[scale=0.5]{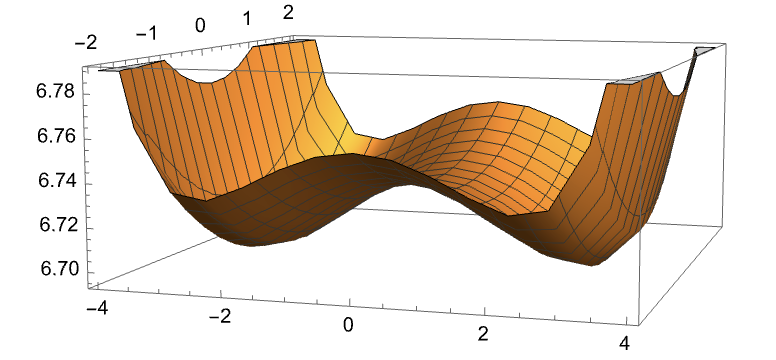}
\caption{Graph of $EH(c_A,c_B)$ when $p=6$ and $q=1$.}
\label{2min1saddle}
\end{figure}

This is in line with the summary above as well as a result from \cite{BoyerHuangLegendre_DH} implying that $EH(c_A,c_B)$ grows without bound as $(c_A,c_B)$ approaches the edges of $D_{p,q}$.
Note that this means
\begin{itemize}
\item For $p\leq 5q$ and $q\leq 5p$, $\EHmin=(\frac{32 (p+q)^3}{p^2 q^2})^{1/3}$.
\item For $p>5q$, $\EHmin=(\frac{432 (p-q)^2}{p^2 q})^{1/3}$.
\item For $q>5p$, $\EHmin=(\frac{432 (p-q)^2}{q^2 p})^{1/3}$.
\end{itemize}

\begin{example}
Suppose $p=6$ and $q=1$. Then $\EHmin=(300)^{1/3} \approx 6.694$.
Let $A(x)=\frac{p^2-x^2}{p}=\frac{36-x^2}{6}$ and $B(y)=\frac{q^2-y^2}{q}=1-y^2$ and consider the corresponding
Sasaki contact form $\eta$.  Evidently, from above we know that $\eta$, despite being cscS, is NOT a minimizer of $EH$ over
$\kt_+$, much less over
$[\eta]^{\bbt^2}$.
As observed in \cite{LahdiliLegendreScarpa_CRYam}, we know that $\caly^{\bbt^2}_{CR}(\eta) \leq  \EHmin$.
We can get a little more information about $\caly^{\bbt^2}_{CR}(\eta)$ by considering $f(x,y)=e^{cx}$, where $c\in \bbr$.
Then, using \eqref{energytrivialhirzebruch}, we get that 
$$EH(A,B,f) = EH(c)= \frac{11 \sinh (12 c)+36 c \cosh (12 c)}{2 \sqrt[3]{6} c \left(\frac{\sinh (18 c)}{c}\right)^{2/3}}.$$
Graphing this function (see Figure~\ref{sharpinequality}) together with $\EHmin \approx 6.694$, it is evident that
we have a strict inequality;
 $\caly^{\bbt^2}_{CR}(\eta) <  \EHmin$.
 What the precise value of  $\caly^{\bbt^2}_{CR}(\eta)$ is remains unknown.
 
\begin{figure}[t]
\includegraphics[scale=0.5]{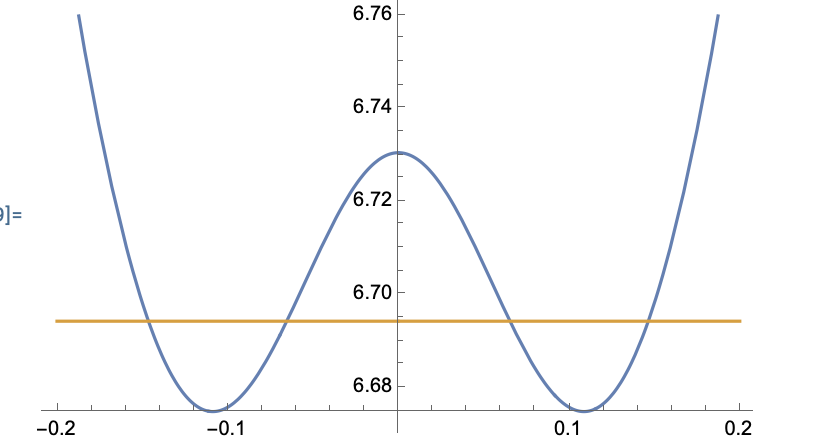}
\caption{Graph of $EH(c)$ when $p=6$ and $q=1$.}
\label{sharpinequality}
\end{figure}

\end{example}

\begin{question}
For $p>5q$, let $$A(x)=\frac{p^2-x^2}{p}+\left(\frac{p - 5 q}{4 p q}\right) \left(\frac{p^2-x^2}{p}\right)^2$$
and $B(y)=\frac{q^2-y^2}{q}$ and consider the corresponding
Sasaki contact form $\eta_\omega$. 
Let $f(x,y)=p^2+q^2+c_A x$ with
$c_A=\pm\frac{\left(p^2+q^2\right) }{p}\sqrt{\frac{p-5 q}{p-q}}$. We know that $\EHmin=EH(f^{-1}\eta_\omega)$ and that
$f^{-1}\eta$ are cscS. Are they are absolute
minimizers of $EH$ over $[\eta]^{\bbt^2}$?
\end{question}

\begin{remark}The behavior of $\bbc\bbp^1\times \bbc\bbp^1$ is remarkably similar for the classical Yamabe problem. Much is still unknown in this setting as well (see the discussions in \cite{LeBrun_EinstMax,lebrun23}), but we do know from
\cite{LeBrun_EinstMax} that if we chose 
$$A(x)=\frac{p^2-x^2}{p}+\left(\frac{c_A^2 p (p+q)}{2 q \left(-c_A^2 p^2+3 p^4+6 p^2 q^2+3 q^4\right)}\right) \left(\frac{p^2-x^2}{p}\right)^2$$
and $B(y)=\frac{q^2-y^2}{q}$, then
for $f(x,y)=p^2+q^2+c_A x$, $Scal(g)_{f,4}=Scal(f^{-2}g)$ is constant
if and only if $c_A=0$ (yielding the usual product cscK metric) or $c_A=\pm \frac{\sqrt{p-2 q} \left(p^2+q^2\right)}{p^{3/2}}$, the latter option being a viable extra solution exactly when 
$p>2q$. Remark~\ref{symmetry} applies here as well.
Now one can consider the Einstein--Hilbert functional associated to the usual Yamabe problem
\begin{equation}\label{yamabetrivialhirzebruch}
EH_{Yam}(f^{-2} g) = 2\pi\frac{\int_{-q}^{q}\int_{-p}^p f^{-4}Scal(f^{-2}g) \,dx\,dy}{\left(\int_{-q}^{q}\int_{-p}^p f^{-4}\,dx\,dy\right)^{1/2}}.
\end{equation}
Restricted to affine linear functions of $x$ and $y$ (and hence considered a function on $D_{p,q}$), this is independent on the choice of K\"ahler metric in the K\"ahler class of $g$ (by \cite{ApostolovMaschler_EinsteinMaxwell} as cited in Proposition 3.3 of \cite{futakiono18}). Indeed, using
for convenience $A(x)=\frac{p^2-x^2}{p}$ and $B(y)=\frac{q^2-y^2}{q}$ we get that with
$f(x,y)=p^2+q^2+c_A x+c_B y$, \tiny
\begin{equation}\label{yamabetrivialhirzebruchaffine}
\begin{array}{cl}
&EH_{Yam}(f^{-2} g)=EH(c_A,c_B) \\
\\
= & 2\pi \sqrt{\frac{48 \left(c_A^2 p^2 (q-p)-q^3 \left(c_B^2-2 p^2\right)+p q^2 \left(c_B^2+2 p^2\right)+p^5+p^4 q+p q^4+q^5\right)^2}{p q \left(-c_A^4 p^4-2 c_A^2 p^2 \left(q (q-c_B)+p^2\right) \left(q (c_B+q)+p^2\right)+\left(q (q-c_B)+p^2\right) \left(q (c_B+q)+p^2\right) \left(c_B^2 q^2+3 \left(p^2+q^2\right)^2\right)\right)}}
\end{array}
\end{equation}\normalsize
As stated in \cite{futakiono18}, for $p\leq 2q$ and $q\leq 2p$, $EH_{Yam}(c_A,c_B)$ has only one critical point, namely at $(c_A,c_B)=(0,0)$. For
$p\leq 2q$ and $q\leq 2p$, this is easily seen to give a
a global minimum of value $\frac{8 \pi  (p+q)}{\sqrt{p q}}$ at $(c_A,c_B)=(0,0)$.
However, as also stated in \cite{futakiono18}, if $p>2q$, we have extra critical points. In fact, in this case, the point $(c_A,c_B)=(0,0)$ is the location of a saddle point and we have
two global minima of value $8 \sqrt{3} \pi  \sqrt{\frac{2 p-q}{p}}$at $(c_A,c_B)=(\pm\frac{\sqrt{p-2 q} \left(p^2+q^2\right)}{p^{3/2}},0)$. Likewise, if
$q>2p$, the point $(c_A,c_B)=(0,0)$ is the location of a saddle point and we have
two global minima of value $8 \sqrt{3} \pi  \sqrt{\frac{2 q-p}{q}}$ at $(c_A,c_B)=(\pm\frac{\sqrt{q-2 p} \left(p^2+q^2\right)}{q^{3/2}},0)$. Thus for $p>2q$ or $q>2p$, the product cscK metric is not a Yamabe minimizer in its conformal class. As discussed in the last section of \cite{LeBrun_EinstMax}, it is not known whether the non-K\"ahler strongly Hermitian
solutions, $f^{-2}g$, of the Einstein-Maxwell equations arising from $c_A=\pm \frac{\sqrt{p-2 q} \left(p^2+q^2\right)}{p^{3/2}}$ when $p>2q$ are
Yamabe minimizers in their conformal class. Indeed, even for $p\leq 2q$ and $q\leq 2p$, we do not know if the product cscK metric (when $p\neq q$) is a Yamabe minimizer in its conformal class. See the ``Technical Question'' in \cite{lebrun23}, which considers the case $p=2q$ (or $q=2p$). However, the observation that they minimize $EH_{Yam}$ over
$D_{p,q}$ is encouraging.
\end{remark}

\subsubsection{Non-Sasaki cscTW examples}
Beyond the cscS examples above, finding explicit solutions to $Scal(A,B,f)=c$, where $Scal(A,B,f)$ is given by \eqref{scaltwtrivialhirzebruch} and $c$ is some constant real number,
is not feasible in general. However, in the following we shall exhibit some explicit non-Sasaki cscTW contact structures 
over $\bbc\bbp^1\times \bbc\bbp^1$. We will show that
none of these are absolute minimizers of $EH$ over $[\eta]^{\bbt^2}$.

Consider the following two functions
$$f_1(x,y)=\frac{1}{A_1+\frac{x^2}{p^2}}\quad \text{and} \quad f_2(x,y)=\frac{1}{A_2-\frac{x^2}{p^2}},$$
where $A_1\in (0,+\infty)\cap \bbq$ and $A_2\in (1,+\infty)\cap \bbq$. By direct calculations, we observe the following:
\begin{enumerate}
\item If $p,q\in\bbz^+$ are co-prime such that $\frac{p}{q}=\frac{\left(103 A_1^2+66 A_1+19\right)}{(A_1+1) (5 A_1+1)}$, 
$$A(x)=\frac{p^2-x^2}{p}-\frac{3 A_1+1}{\left(103 A_1^2+66 A_1+19\right) q} \left(\frac{p^2-x^2}{p}\right)^2,$$ and
$B(y)=\frac{q^2-y^2}{q}$, then
$$Scal(A,B,f_1) =\frac{72 (3 A_1+1)}{\left(103 A_1^2+66 A_1+19\right) q},$$
$$EH(A,B,f_1)=
\left(\frac{1492992 (3 A_1+1)^3 \left(35 A_1^3+35 A_1^2+21 A_1+5\right)}{35 (A_1+1) (5 A_1+1) \left(103 A_1^2+66 A_1+19\right)^2 q}\right)^{1/3},$$
and, since $p>5q$,
$$\EHmin=\left(\frac{432 (p-q)^2}{p^2 q}\right)^{1/3}=\left(\frac{1728 \left(49 A_1^2+30 A_1+9\right)^2}{\left(103 A_1^2+66 A_1+19\right)^2 q}\right)^{1/3}.$$
Finally, 
$$
\begin{array}{cl}
&(EH(A,B,f_1))^3-(\EHmin)^3\\
\\
= & \frac{1728 \left(396305 A_1^6+614250 A_1^5+565243 A_1^4+337308 A_1^3+120231 
A_1^2+21114 A_1+1485\right)}{35 (A_1+1) (5 A_1+1) \left(103 A_1^2+66 A_1+19\right)^2 q}\\
\\
> & 0,
\end{array}
$$
and hence $f_1$ does not give an absolute minimizer of $EH$ over the $\bbt^2$-invariant conformal class of contact forms.

\item If $p,q\in\bbz^+$ are co-prime such that $\frac{p}{q}=\frac{\left(103 A_2^2-66 A_2+19\right)}{(A_2-1) (5 A_2-1)}$, 
$$A(x)=\frac{p^2-x^2}{p}+\frac{3 A_2-1}{\left(103 A_2^2-66 A_2+19\right) q}\left(\frac{p^2-x^2}{p}\right)^2,$$ and
$B(y)=\frac{q^2-y^2}{q}$, then
$$Scal(A,B,f_2) =\frac{72 (3 A_2-1)}{\left(103 A_2^2-66 A_2+19\right) q},$$
$$EH(A, B,f_2)=
\left(\frac{1492992 (3 A_2-1)^3 \left(35 A_2^3-35 A_2^2+21 A_2-5\right)}{35 (A_2-1) (5 A_2-1) 
\left(103 A_2^2-66 A_2+19\right)^2 q}\right)^{1/3},$$
and, since $p>5q$,
$$\EHmin=\left(\frac{432 (p-q)^2}{p^2 q}\right)^{1/3}=\left(\frac{1728 \left(49 A_2^2-30 A_2+9\right)^2}{\left(103 A_2^2-66 A_2+19\right)^2 q}\right)^{1/3}.$$
Finally, 
$$
\begin{array}{cl}
&(EH(A,B,f_2))^3-(\EHmin)^3\\
\\
= & \frac{1728 \left(396305 A_2^6-614250 A_2^5+565243 A_2^4-337308 A_2^3+120231 A_2^2-21114 A_2+1485\right)}{35 (A_2-1)
 (5 A_2-1) \left(103 A_2^2-66 A_2+19\right)^2 q}\\
 \\
 =& \frac{1728 \left( 396305 (A_2-1)^6+1763580 (A_2-1)^5+3438568 (A_2-1)^4+3707264 (A_2-1)^3  \right)}
 {35 (A_2-1) (5 A_2-1) \left(103 A_2^2-66 A_2+19\right)^2 q}\\
 \\
 +& \frac{1728 \left( 2301840 (A_2-1)^2+774976 (A_2-1)+110592  \right)}
 {35 (A_2-1) (5 A_2-1) \left(103 A_2^2-66 A_2+19\right)^2 q}\\
 \\
> &0,
\end{array}
$$
and hence $f_2$ does not give an absolute minimizer of $EH$ over the $\bbt^2$-invariant conformal class of contact forms.
\end{enumerate}

\subsection{Ruled surfaces of the form \texorpdfstring{$\bbp(\BOne \oplus {\call}_k)\rightarrow \Sigma$}{P(1+Lk)}}\label{sec:proj_bundle}

Let $X$ be (the total space of) the bundle $\bbp(\BOne \oplus {\call}_k)\rightarrow \Sigma$, where $\Sigma$ is a compact Riemann surface of genus $\gg$, $\BOne$ is the trivial holomorphic line bundle over 
$\Sigma$, and ${\call}_k \rightarrow \Sigma$ is a holomorphic line bundle over $\Sigma$ such that $c_{1}({\call}_k) =\left[\frac{\omega_{\Sigma}}{2 \pi}\right]$,
where $(\omega_\Sigma,g_\Sigma)$ is a KE structure on $\Sigma$ with scalar curvature $2s=\frac{4(1-\gg)}{k}$ for
$k\in \bbz^+$. We say that ${\call}_k$ has degree $k$.

In \cite{Calabi_extremal_metrics} Calabi constructed extremal K\"ahler metrics in every K\"ahler class of $X$ when $\Sigma=\bbc\bbp^1$ and these metrics are each a special case of \emph{admissible K\"ahler metrics} as defined in \cite{HamFormsIII}. 
First, we give a quick overview of such admissible metrics on $X$ (very similar to e.g. Section 3.1 of \cite{BHLT25}). 

\subsubsection{Admissible K\"ahler metrics on \texorpdfstring{$\bbp(\BOne \oplus {\call}_k)\rightarrow \Sigma$}{P(1+Lk)}}\label{admreview}

We will consider the $\bbc^*$-action on $X$,  defined by diagonal multiplication on $\calo\rightarrow \Sigma$ and denote by $X^0$ the open dense subset of regular points of the action. Note that $X^0$ has the structure of a principal $\bbc^*$-bundle over the (stable) quotient under the $\bbc^*$-action of $X$. This corresponds to the $\bbc^*$-bundle over $\Sigma$,  obtained from the $\bbc \bbp^1$-bundle $X \rightarrow \Sigma$ by deleting the zero and infinity sections $E_0:=P(\BOne \oplus 0)$ and $E_\infty:=P(0\oplus {\call}_k)$.

There exist Hermitian metrics $h_0$ on $\BOne$ and $h_{\infty}$ on ${\call}_k$ whose  respective Chern connections  have curvatures
$0$ and  $\omega_{\Sigma}$, respectively. Let $r_0$ and $r_{\infty}$ denote the corresponding fibre-wise norm functions.
We denote the generator of the circle action on $\BOne$  by $K_{0}$  and the generator of the circle action on $\call_k$
by $K_{\infty}$. Using the Chern connections of $(\calo,  h_0)$  and $(\calo(k),h_{\infty})$, we let $\theta_0$ and $\theta_{\infty}$ be the connection 1-forms defined on the corresponding unitary bundles, i.e. satisfying
\begin{equation*}
\begin{split}
 \theta_0(K_0) &=1,  \ \ d\theta_0 = 0; \\
\theta_{\infty} (K_{\infty})&=1,  \ \ d\theta_{\infty} = -\omega_{\bbc \bbp^1}.
 \end{split}
 \end{equation*}
Thus, the fibre-wise Euclidean structures (viewed as tensors on the total spaces $\BOne$ and $\call_k$) take the following momentum/angular form
\begin{equation*}
g_0=\frac{{dz_0}\otimes d\gz_0}{2\gz_0} + 2\gz_0 ( \theta_0\otimes \theta_0), \ \ g_{\infty}=\frac{d\gz_{\infty}\otimes d\gz_{\infty}}{2\gz_{\infty}} + 2\gz_{\infty} (\theta_{\infty}\otimes \theta_{\infty}),
\end{equation*}
where  $\gz_0:= r^2_0/2,  \gz_{\infty} := r_{\infty}^2/2$ are the fibre-wise momentum coordinates.

Let $0<x<1$  be a fixed real number. We then  consider the smooth positive semidefinite tensor  on the total space $\BOne \oplus {\call}_k$:
\begin{equation*}
\frac{(1+x)\gz_0+ (1-x)\gz_{\infty}}{2x}g_{\Sigma} + g_{0} + g_{\infty}.
\end{equation*}
Considering the ``K\"ahler quotient''  for this tensor with respect to the $S^1$-action generated by $K_0 + K_{\infty}$ at the  level set $\gz_0 + \gz_{\infty}=2$ on $\BOne \oplus {\call}_k$, we denote by  $g_c$ the smooth (possibly degenerate) tensor field  induced on 
$X$ and by $\omega = g_c J_c$ the corresponding smooth $(1,1)$-form, where $J_c$ is the induced (canonical) complex structure on $X$. Letting $\gz:=(\gz_0-\gz_{\infty})/2\in [-1,1]$, the degenerate K\"ahler structure $(g_c, \omega)$ is written  on $X^0$ as:
\begin{equation}\label{g}
g_c=\frac{1+x\gz}{x}g_{\Sigma}+\frac {d\gz^2}
{\Theta_c (\gz)}+\Theta_c (\gz)\theta^2,\quad
\omega_x = \frac{1+x \gz}{x}\omega_{\Sigma} + d\gz \wedge
\theta,
\end{equation}
where $\Theta_c (\gz)= 1-\gz^2$ and $\theta := {\theta}_{0}- {\theta}_{\infty}$ satisfies
\begin{equation}\label{theta}
d\theta = \omega_{\Sigma}.
\end{equation}
We notice that $\gz$ is the momentum map with respect to $\omega$ of the induced $S^1$-action on $X$ corresponding to multiplication  on $\BOne$ or,  equivalently, the $S^1$-action induced by the push forward of $K=(K_0-K_{\infty})/2$ to the quotient 
space $X$. Thus,  $E_\infty = \gz^{-1}(-1), E_{0} = \gz^{-1}(1),$ and $X^0 = \gz^{-1}(-1,1)$. It follows that $(g_c, \omega)$ defines a K\"ahler metric on $X^0$. It is shown in \cite{HamFormsIII} that $(g_c, \omega)$ gives rise to a genuine, non-degenerate, smooth K\"ahler metric on $X$. Then, $\gz$ is the momentum map  with respect to $\omega$  of the $S^1$-action on $M$ by multiplication on  $E_0$. In particular, $\gz$ is a Killing potential and hence any affine function of $\gz$ is a Killing potential.

As observed in \cite{HamFormsIII}, due to \cite{Hwang_Singer}, if instead of $\Theta_c(\gz)$ we take in \eqref{g} any smooth function $\Theta(\gz)$ on $[-1,1]$, satisfying
\begin{align}
\label{positivity}
(i)\ \Theta(\gz) > 0, \quad -1 < \gz <1,\quad
(ii)\ \Theta(\pm 1) = 0,\quad
(iii)\ \Theta'(\pm 1) = \mp 2.
\end{align}
then
\begin{equation}\label{metric}
g_x=\frac{1+x\gz}{x}g_{\Sigma}+\frac {d\gz^2}
{\Theta (\gz)}+\Theta (\gz)\theta^2,\quad
\omega_x =\frac{1+x\gz}{x}\omega_{\Sigma} + d\gz \wedge
\theta,
\end{equation}
with \eqref{theta} yields a smooth $S^1$-invariant K\"ahler metric on $X$, compatible with the same symplectic form $\omega$. The corresponding  complex structure  is then given on $X^0$ by the horizontal lift  of the base complex structure on $\Sigma$ (with respect to the chosen Chern connections on $\BOne$ and ${\call}_k$) along with the requirement 
\begin{equation}\label{complex}
Jd\gz = \Theta \theta
\end{equation}
on the fibres. Such K\"ahler metrics on $X$ are called {\it admissible K\"ahler metrics}.

It is convenient to define a function $F(\mathfrak{z})$ by the formula
\begin{equation}
\Theta(\mathfrak{z})= \frac{F(\mathfrak{z})}{(1+x
\mathfrak{z})}.
\end{equation}
Since $(1+x
\mathfrak{z})$ is positive for $-1<\mathfrak{z}<1$, conditions
\eqref{positivity}
imply the following equivalent conditions on $F(\mathfrak{z})$:
\begin{align}
\label{positivityF}
(i)\ F(\mathfrak{z}) > 0, \quad -1 < \mathfrak{z} <1,\quad
(ii)\ F(\pm 1) = 0,\quad
(iii)\ F'(\pm 1) = \mp 2(1 \pm x).
\end{align}

Let $\mathcal{K}^{\rm adm}(X,\omega_x)$ denote the space of all admissible K\"ahler metrics associated to a given choice of $x\in (0,1)$. From the discussion above, $\mathcal{K}^{\rm adm}(X,\omega_x)$ is identified with a Fr\'echet space consisting of all smooth functions $\Theta(z)$ on $[-1,1]$ satisfying \eqref{positivity}. The space $\mathcal{K}^{\rm adm}(X, \omega_x)$  (associated  to a given choice of $x\in(0,1)$)  can also be equally parametrized by the fibre-wise symplectic potentials $u(\gz)$, where $u(\gz)$ is defined up to an affine-linear  function of $\gz$ by  $u''(\gz)= \frac{1}{\Theta(\gz)}$. 

One easily checks that the K\"ahler class of an admissible metric \eqref{metric} is given by
\begin{equation}\label{class1}
[\omega_x] = 4\pi E_{0}^*+ \frac{2\pi(1-x)k}{x} C^*,
\end{equation}
where $C$ denotes a fiber of the ruling $X \rightarrow
\Sigma$,
Up to an overall rescale, every K\"ahler class in the K\"ahler cone may be represented by an admissible K\"ahler metric. 

The scalar curvature of an admissible metric $g_x$ as above is given by 
\begin{equation}\label{admscal}
Scal_x=\frac{2sx}{1+x\gz}-\frac{F''(\gz)}{1+x\gz}.
\end{equation}
If $u=u(\gz)$ is a smooth function $\gz$ (and hence $u$ can be viewed as a smooth function on $X$) and if $\Delta_{g_x}$ denotes the Laplacian of 
$g_x$, then
\begin{equation}\label{admlaplace}
\Delta_x u =-\frac{\frac{d}{d\gz}\left[F(\gz)u'(\gz)\right]}{1+x\gz}.
\end{equation}

\subsubsection{CR Einstein--Hilbert functional on the non-trivial Hirzebruch surfaces}
\label{ex:negativeHirzebruchsurf}
We assume now that $0<s\leq 2$ and that 
$x\in \bbq\cap (0,1)$ so that up to a rescale, the K\"ahler class $[\omega_x]$ is integer and we have a Sasaki manifold over the first Hirzebruch surface via the Boothby-Wang construction.
Suppose $F(\gz)=(1-\gz^2)(1+x\gz)$ and $f(\gz)=c\gz+1$ is a killing potential, where $-1<c<1$. Up to an overall scale, the CR Einstein--Hilbert functional is then given by
$$
\begin{array}{ccl}
EH(x,F,f) &=& \frac{\int_{-1}^1 f^{-3}\left(f Scal_x-6 \Delta_g f - 12|df|^2_g f^{-1}\right)(1+x\gz)d\gz}{\left(\int_{-1}^1 f^{-3}(1+x\gz)d\gz\right)^{2/3}}\\
\\
&=& \frac{\int_{-1}^1 f^{-2}\left(2sx-F''(\gz)+6 \left(\frac{f'(\gz)}{f(\gz)}\right)^2 F(\gz)\right)d\gz}{\left(\int_{-1}^1 f^{-3}(1+x\gz)d\gz\right)^{2/3}}\\
\\
&=&\frac{2 \left(\frac{2(1- c x)}{\left(1-c^2\right)^2}\right)^{1/3} \left(1 + c^2 - 2 c x +(1-c^2)x s\right)}{1-c x},
\end{array}
$$
which is clearly positive. 
Indeed, as a function of $-1<c<1$ (representing (up to scale) a sub-cone of the Sasaki cone), this has a global positive minimum when
$-1<c<1$ is the unique solution, $c_x$ to
$$(sx-2) x+ (5 - s x + x^2)c - x (6 + sx)c^2+(3x^2+sx-1)c^3=0.$$
Note that $c_x\neq x$ and hence for the solution $c_x$ given by a choice of $s$ and $x$ we have that
$$s=\frac{3 c_x^3 x^2-c_x^3-6 c_x^2 x+c_x x^2+5 c_x-2 x}{(1-c_x^2) x (c_x-x)}$$
If we then choose $F=F_x$, given by
$$F_x(\gz) \underbrace{=}_{\text{Sec. 3.2 in \cite{BHLT25}}}F_{x,c_x}(\gz) \underbrace{=}_{\text{eliminating $s$}} \frac{(1-\gz^2) (1 + c_x \gz) (1 - c_x x - c_x \gz + x \gz)}{1-c_x^2},$$
and $f_x(\gz)=c_x\gz+1$ we have that $(g_x,f_x)$ is $4$-weighted extremal with $Scal_{f_x,4}$ equal to a constant time $f_x$ and hence this corresponds to a cscS metric in the Sasaki cone. Due to the existence of this positive cscS, we can also deduce that $EH$ has to be positive over the entire Sasaki cone. 

Now consider the deformation of $F_x(\gz)$ given by
$$F_a(\gz)=F_x(\gz) +a (1 - \gz^2)^2 (1 + x \gz), $$
where $a\geq 0$. Consider $\hat{f}(\gz)=2-\gz^2$. This is a positive smooth function that is decidedly not a killing potential. One may now calculate that up the the same overall scale as above, the CR Einstein--Hilbert functional is
$$
\begin{array}{cl}
&EH(x,F_a,\hat{f})\\
\\
=& \text{a value depending only on $x$ and $s$ (and hence $c_x$)}\\
\\
+& \left(\frac{18 \sqrt[3]{2}-19\cdot 2^{5/6} \sinh ^{-1}(1)}{\left(14+3 \sqrt{2} \sinh ^{-1}(1)\right)^{2/3}}\right)a.
\end{array}
$$
Since $\left(\frac{18 \sqrt[3]{2}-19\cdot 2^{5/6} \sinh ^{-1}(1)}{\left(14+3 \sqrt{2} \sinh ^{-1}(1)\right)^{2/3}}\right)\approx -1.0526$ is negative, this will be negative for sufficiently large $a$ values.

Note that similar observations can be made for higher genus ruled surfaces of Hirzebruch type, but since that case is not toric, we focus on the Hirzebruch surfaces here.

Likewise, similar observations can be made for the trivial Hirzebruch surface in \S\ref{sec:trivialHirzebruch} where we can deform the cscS solutions in the K\"ahler class and then
(when deformed sufficiently), the EH will be negative for a a certain non-Killing potential positive smooth function. 

This is all to demonstrate that negative energy sufficiently far away from positive cscS solutions seems easy to obtain explicitly.

\subsection{Non-Toric Examples}
\label{sec:nontoric}
In this section we will consider some non-toric examples. 
The first example will be $7$-dimensional and the second example will be $9$-dimensional. For each of these examples the Sasaki Reeb cone is just $2$-dimensional and thus the Sasaki rays may be parametrized by one real parameter. These examples are meant as a supplement to Question~\ref{q:uniqueness}. In particular, they give further evidence that the cscS condition in itself is certainly not sufficient for being a $CR$ $\bbt$-Yamabe minimizer...even if the Sasaki cone is just two dimensional. The second example, Example~\ref{multicscS2}, also shows that
it is possible for a cscS solution to be a local minimizer of the CR Einstein--Hilbert functional over the Sasaki cone and yet not be a $CR$ $\bbt$-Yamabe minimizer.

\begin{example}\label{multicscS1}
Consider the product $(\bbc\bbp^2\#q\overline{\bbc\bbp}^2) \times \bbc\bbp^1$
for $4 \leq q \leq 8$ (where $\bbc\bbp^2$ is blown up at
$q$ generic points). As is well known, $\bbc\bbp^2\#q\overline{\bbc\bbp}^2 $, which has complex dimension $3$, admits a KE structure $(g_k,\omega_k)$
which, for any choice of $k\in \bbq^+$, can be rescaled such that the scalar curvature equals $k$. 
Now we consider the product K\"ahler structure $(g,\omega):=(g_k+g_F,\omega_k+\omega_{\bbc\bbp^1})$ on $(\bbc\bbp^2\#q\overline{\bbc\bbp}^2)\times \bbc\bbp^1$, where
$(g_F,\omega_{\bbc\bbp^1})$ is a K\"ahler structure on $ \bbc\bbp^1$ given by the standard toric set-up as in
\eqref{cp1toric} and \eqref{endpointscp1}.

We may think of $\gz:\bbc\bbp^1 \rightarrow [-1,1]$ as a moment map of a choice of a $S^1$ action on $\bbc\bbp^1$ and $\omega_{\bbc\bbp^1}$
and then the pull-back of $\gz$ to $(\bbc\bbp^2\#q\overline{\bbc\bbp}^2 ) \times \bbc\bbp^1$ (with K\"ahler structure $(g,\omega)$) is a Killing potential.

Since $k\in \bbq^+$, we have that $[\omega_k/2\pi]$ is a rational K\"ahler class.
Thus, the class $[\left(\omega_k+\omega_{\bbc\bbp^1}\right)/2\pi]$ is also a rational K\"ahler class and hence determines a canonical primitive polarization which via a Boothby-Wang construction gives a Sasaki manifold $(N,D,\eta_{\omega},J_F)$, where the choice of $F$ determines the $CR$-structure on $D$.
By considering different values of $k\in \bbq^+$, we are implicitly considering all the possible polarization choices for 
$(\bbc\bbp^2\#q\overline{\bbc\bbp}^2) \times \bbc\bbp^1$. 

The scalar curvature of $g$ is given by 
$Scal=k-F''(\gz)$ (pulled back to the product) and if $f=f(\gz)$ is a smooth function of $\gz$ (and hence $f$ can be viewed as a smooth function on $(\bbc\bbp^2\#q\overline{\bbc\bbp}^2) \times \bbc\bbp^1$)
and if $\Delta_{g}$ denotes the Laplacian of 
$g$, then $\Delta_g f =-\frac{d}{d\gz}\left[F(\gz)f'(\gz)\right]$.

We are now ready to write up $Scal_{f,p}(g)$ for $p=2+1+2=5$, $f_c= c\gz+1$, and $-1<c<1$:
\begin{equation}
Scal_{f_c,5}(g)= (1+c\gz)^2 \left(k-F''(\gz)\right)  +8c (1+c\gz)F'(\gz) - 20c^2F(\gz).
\end{equation}
First we are looking for extremal Sasaki metrics, that is, solutions to 
$$ (1+c\gz)^2 \left(k-F''(\gz)\right)  +8c (1+c\gz)F'(\gz) - 20c^2F(\gz)=A\gz+B$$
where $A,B$ are real constants.
This can be rewritten as
$$-(1+c\gz)^2 F''(\gz) +8c (1+c\gz)F'(\gz) - 20c^2F(\gz)=-k(1+c\gz)^2+A\gz+B$$
or
$$\frac{d^2}{d\gz^2}\left[\frac{F(\gz)}{(c\gz+1)^4}\right]=\left(k(1+c\gz)^2-A\gz-B\right)(c\gz+1)^{-6},$$
which, due to the endpoint conditions on $F(\gz)$, determines ($A$, $B$ and) $F(\gz)$ in terms of $k$, and $c$. 
We have
$$A=\frac{2 c \left((k-6)c^4 k-(6k+12) c^2 +5 k-30\right)}{c^4-2 c^2+5},$$
$$ B =\frac{(5k-30) c^6 k-(19k+6) c^4 -(282+k)c^2 +15 k+30}{3 \left(c^4-2 c^2+5\right)},$$
and
$$F(\gz)=\frac{(1-\gz^2) \left( 6 \left(c^4 (\gz^2+1)+c^2 +10-2 c^3 \gz^3+2 c^3 \gz-5 c^2 \gz^2\right)+k c^2 (1-\gz^2) (5 - c^2 +2 c \gz) \right)}
{12 \left(c^4-2 c^2+5\right)}$$
which satisfies the required $F(\gz) > 0$ for $-1 < \gz <1$. At the Sasaki level, this means that the entire Sasaki cone (which is $2$-dimensional here) is exhausted by
extremal Sasaki metrics. Among these we can now look for cscS:

It is easy to see that that $A\gz+B$ is a constant multiple of $c\gz+1$ (meaning the  Reeb vector field associated to $c$ is cscS) if and only if
$A-Bc=0$. Using the above formula for $A$ and $B$, this condition is equivalent to
\begin{equation}\label{7DcscS}
c (3(14 - k) +4 k c^2  +(6-k) c^4)=0.
\end{equation}
Note that the solution $c=0$ corresponds to the case of the transverse product metric on $(\bbc\bbp^2\#q\overline{\bbc\bbp}^2) \times \bbc\bbp^1$
being a product of cscK metric.

Now we will use our extremal Sasaki solutions to consider the CR Einstein--Hilbert functional restricted to the Sasaki cone (parametrized up to scale by $c\in (-1,1)$), $EH(c):=
EH(k,F,f_c)$. Up to an overall scale we have
$$EH(c)=\frac{\int_{-1}^1 (A z + B) (c z + 1)^{-5}d\gz}{\left(\int_{-1}^1(c z + 1)^{-4}d\gz\right)^{\frac{3}{4}}}
=\frac{\sqrt[4]{54} \left(c^2 (6-k) +(k+2)\right)}{(1-c^2)^{3/4} (c^2+3)^{3/4}}.$$
Naturally, the critical points of $EH(c)$ correspond to the (cscS) solutions $c\in (-1,1)$ to \eqref{7DcscS} and it is not hard to see that for $k\in \bbq \cap (0,14]$, we have just one critical point, $c=0$, giving a global minimum on the Sasaki cone, and for $k\in \bbq \cap (14,+\infty)$ we have three critical points, $c=0$, and $c=\pm \sqrt{\frac{2 k-\sqrt{k^2+60 k-252}}{k-6}}$. In this case $EH(c)$ has a local maximum at $c=0$ whereas the other critical points are the locations of the global minimum value of the CR Einstein--Hilbert functional over the Sasaki cone. This example represent a simple non-toric case of a Sasaki manifolds with several cscS rays in the Sasaki cone and with one of these cscS (when $c=0$) decidedly not being a $CR$-Yamabe minimizer.
\end{example}

\begin{remark}
Note that we can view the Sasaki manifold considered in Example~\ref{multicscS1} as a so-called \cite{BoTo16,BoTo22} $S^3$-join, $M\star_{l_1,l_2}S^3$, where $M$ is the Boothby-Wang constructed Sasaki manifold over (an appropriate rescale of $(\bbc\bbp^2\#q\overline{\bbc\bbp}^2,\omega_k,g_k)$, $S^3$ is the standard Sasaki structure over $\bbc\bbp^1$, and $l_1,l_2$ are co-prime positive integers implicitly determined by the value of $k$. The value of $l_2$ will grow as the value of $k$ grows. In this light, existence of multiple cscS rays when $k$ is sufficiently large is predicted by Section 6.4 (and Theorems 1.2 and 1.3) of \cite{BoTo16}. 
\end{remark}

\begin{example}\label{multicscS2}
We will now take a close look at Example 4.6 from \cite{BoTo26}. For the convenience of the reader we will repeat some of the details from \cite{BoTo26}. After this we will add an analysis of the CR Einstein--Hilbert functional in this case. We use the set-up in \S\ref{admreview} as well
as a basic product metric understanding. Assume that $N_1 = \bbc\bbp^2\#q\overline{\bbc\bbp}^2$
for $4 \leq q \leq 8$ ($\bbc\bbp^2$ blown up at
$q$ generic points). Then $N_1$ admits a KE metric $g_a$
which, for any choice of $a\in \bbq^+$, can be rescaled such that the scalar curvature equals $2a$. Note that $N_1$ admits no non-constant Killing potentials.
Let $N_2$ equal the total space of $\bbp(\BOne \oplus {\call}_k)\rightarrow \Sigma$ as in \S\ref{admreview} and assume that the genus of $N_2$, $\gg$, is at least two.

Given $a\in \bbq^+$, $x\in (0,1)\cap \bbq$, $s=\frac{2(1-\gg)}{k}$, and the K\"ahler structure on $N_2$ given by \eqref{metric}, we consider the product K\"ahler metric $g=g_a+g_x$ on $N_1\times N_2$
with product K\"ahler form $\omega=\omega_a+\omega_x$. After an appropriate rescaling this gives us a polarization of 
$N_1\times N_2$ yielding a $9$-dimensional Sasaki manifold with a $2$-dimensional Sasaki cone.
Specifically, for $-1<c<1$, the pull-back of $f=1+c\gz$ to $N_1\times N_2$ is a positive smooth real Killing potential and defines an element in the 
(unreduced) Sasaki cone. Since $N_1$ allows no non-trivial Killing potentials and since $\gg\geq 2$, we see that, up to scale,
$-1<c<1$ parametrizes the Sasaki cone. Since the complex dimension of $N_1\times N_2$ is $4$, we get that ($p=6$ and)
$$
f Scal^{TW}(a,s,x,F,f)=Scal_{f,6}(g)= f^2 Scal(g)  -10 f \Delta_g f - 30 |df|^2_g.
$$
Note that for any smooth, positive, real function $f$ on $N_2$, viewed via a pull-back as a function on $N_1\times N_2$, also denoted $f$, we have that
\begin{equation}\label{scal9dimex}
Scal_{f,6}(g)=\frac{ f^2 \left(2a (1+x\gz)+2 s x - F'' \right)  + 10 f \frac{d}{d\gz}(F\, f')- 30 F (f')^2}{(1+x\gz)},
\end{equation}
where $F(\gz)=(1+x\gz)\Theta(\gz)$ determines $g_x$ in \eqref{metric}.
Using the killing potential $f=1+c\gz$, $Scal_{f,6}(g)$ is then a killing potential $A_1\gz+A_2$ (leading to a Sasaki extremal ray) if and only if $F(\gz)$ is given by
\begin{equation}\label{thisisF9dimex}
F(\gz)= (c\gz+1)^5\left(\frac{2(1-x)}{(1-c)^{5}}(\gz+1) + \int_{-1}^{\gz}Q(t)(\gz-t)dt\right),
\end{equation}
where 
$$
Q(\gz)=(c\gz+1)^{-7}\left((c\gz+1)^2\left(2a(1+x\gz)+2sx\right)-(A_1\gz + A_2)(1+x\gz)\right),
$$
and $A_1$, $A_2$ are solutions to the system
\begin{equation}\label{A1A2}
\begin{array}{ccl}
\alpha_{1,-7} A_1+\alpha_{0,-7} A_2 &=&2\beta_{0,-5}\\
\\
\alpha_{2,-7} A_1+\alpha_{1,-7} A_2 &=&2\beta_{1,-5},
\end{array}
\end{equation}
with 
$$
\alpha_{r,q}=\int_{-1}^1 (ct+1)^q t^r(1+xt)dt
$$
and
$$
\beta_{r,q}=\int_{-1}^1\left(a(1+xt)+sx\right) t^r(ct+1)^qdt +\left( (-1)^r(1-c)^q(1-x)+(1+c)^q(1+x)\right),
$$
Finally, $A_1\gz+A_2$ is a constant multiple of $c\gz+1$ (leading to a cscS ray) precisely when
\begin{equation}\label{cscS9dimex}
\alpha_{1,-6}\beta_{0,-5}-\alpha_{0,-6}\beta_{1,-(5)}=0.
\end{equation}

We now choose $k$ and $\gg$ such that $s=-3$ (e.g. we could have $k=2$ and $\gg=4$) and let $a$ and $x$ be chosen such that
$a=\frac{3 \left(x^4+7\right)}{\left(1-x^2\right) \left(3-x^2\right)}$. In that case, the left hand side of \eqref{cscS9dimex} equals
$$\frac{4 (x-c) h(x,c)}{15 (1-c^2)^9  (1-x^2)  \left(3-x^2\right)},$$
where 
$$
\begin{array}{ccl}
h(x,c)& = & (-12 + 9 x + 8 x^2 - 12 x^3 + 20 x^4 + 3 x^5) c^6\\
&+ &2 x (9 - 16 x^2 + 15 x^4)c^5\\
&+ & (-18 - 27 x - 72 x^2 + 36 x^3 - 38 x^4 - 9 x^5)c^4\\
&+& 8 x (5 - x^2) (9 - 7 x^2)c^3\\
&+ &(-228 + 63 x + 408 x^2 - 84 x^3 - 68 x^4 + 21 x^5)c^2\\
&- & 10 x (9-x^2) c\\
& - &5 (1 -x^2 ) (2 + 3 x) (3 - x^2).
\end{array}
$$
Thus \eqref{cscS9dimex} is always solved by $c=x$ and any solutions beyond this, would be roots $-1<c<1$ of $h(x,c)$ as a polynomial in $c$.
First we note that with $c=c_1:=x$, $a=\frac{3 \left(x^4+7\right)}{\left(1-x^2\right) \left(3-x^2\right)}$, and $s=-3$, equation \eqref{thisisF9dimex}
gives us that 
\begin{equation}\label{9dimexquasiregcscS}
F(\gz):=F_{c_1}(\gz) =\frac{ (1 - \gz^2) (1 + x \gz)^2 (3 + x^2 - x(3-x^2) \gz  - 
    2 x^2 \gz^2)}{(1 - x^2)(3 - x^2)},
    \end{equation}
 which is easily seen to satisfy  {\em (i)} of \eqref{positivityF}. Thus we have at least one cscS ray in the Sasaki cone.
As we vary the value of $0<x<1$ we notice that the nature of $h(x,c)$ will change. If $x\in (0,1)\cap \bbq$ is sufficiently small, $h(x,c)<0$ for $-1<c<1$ and thus no further cscS rays exist in the 
corresponding Sasaki cone. On the other hand e.g. 
$$
\begin{array}{ccl}
h(8/10,c)&= &\frac{2 \left(5236 c^6+12260 c^5-134003 c^4+197072 c^3+30532 c^2-107380 c-29205\right)}{3125}\\
\\
\text{and}\\
\\
h(9/10,c)&=& \frac{3 \left(290849 c^6+352890 c^5-3487407 c^4+3348648 c^3+2190983 c^2-2503170 c-325945\right)}{100000}
\end{array}
$$
each have two roots in $(-1,0)$. This can for example be seen directly by noticing that for both $x=8/10$ and $x=9/10$ we have $h(x,0)<0$, $h(x,-2/5)>0$, and $h(x,-9/10)<0$.

Now, for $x=8/10$, $a=\frac{3 \left(x^4+7\right)}{\left(1-x^2\right) \left(3-x^2\right)}=4631/177$, and $s=-3$, $F(\gz)$ given by \eqref{thisisF9dimex}, 
is numerically seen to satisfy  {\em (i)} of \eqref{positivityF} for all values of $c\in (-1,1)$. This means that the Sasaki cone is exhausted by extremal Sasaki rays and we
have three distinct cscS rays in this cone.

\begin{remark} For $x=9/10$,  $a=\frac{3 \left(x^4+7\right)}{\left(1-x^2\right) \left(3-x^2\right)}= 76561/1387$, and $s=-3$, $F(\gz)$ given by \eqref{thisisF9dimex}, 
does not always satisfy  {\em (i)} of \eqref{positivityF}. Numerically it can be seen that there exists a small closed interval $I=(c_l,c_r) \subset (-1,1)$ with $c_l<0<c_r$
such that for
$c\in I$, $F(\gz)$ does not satisfy {\em (i)} of \eqref{positivityF} while for $c\in (-1,1)\setminus I$, $F(\gz)$ satisfies {\em (i)} of \eqref{positivityF}. 
Thus, due to Theorem 3/Corollary 7.3 in \cite{ApostolovCalderbankLegendre_weighted}, the extremal Sasaki cone is disconnected
and is represented up to scale by two open intervals $I_l=(-1,c_l)$ and $I_r=(c_r,1)$.
Any ray in the ``moat'' determined by $c\in I$ has no extremal Sasaki metric at all. Finally, numerically the roots of $h(9/10,c)$ in $(-1,0)$ are
$c_2\approx -0.120$ and $c_3\approx -0.786$. It can be checked that $c_1=9/10\in I_r$, $c_2\in I$, and $c_3\in I_l$. Thus $c_1$ and $c_3$ correspond to genuine cscS rays, whereas $c_2$ does not.
The two cscS rays are thus separated by a non-extremal moat in the Sasaki cone.
\end{remark}

 We now use \eqref{scal9dimex} to see that for any smooth, positive, real function $f$ on $N_2$, viewed via a pull-back as a function on $N_1\times N_2$, also denoted $f$, we have that, up to scale,

\begin{equation}
\begin{array}{ccl}
EH(a,s,x,F,f)&=& \frac{\int_{-1}^{1}f(\gz)^{-6} Scal_{f,6}(g) (1+x\gz)d\gz}{\left(\int_{-1}^1f(\gz)^{-5}(1+x\gz)d\gz\right)^{\frac{4}{5}}}\\

\\
&=& \frac{\int_{-1}^{1}f(\gz)^{-6} \left(f(\gz)^2 \left(2a(1+x\gz)+2sx -F''(\gz)\right)  +10 f(\gz)\frac{d}{dz}\left(F(\gz)f'(\gz)\right) - 30(f'(\gz))^2 F(\gz) \right)d\gz}{\left(\int_{-1}^1f(\gz)^{-5}(1+x\gz)d\gz\right)^{\frac{4}{5}}}\\
\\
&=& \frac{\int_{-1}^{1} \left(f(\gz)^{-4} \left(2a(1+x\gz)+2sx -F''(\gz)\right)   +10f(\gz)^{-5}\frac{d}{dz}\left(F(\gz)f'(\gz)\right) - 30\frac{(f'(\gz))^2}{f(\gz)^6}F(\gz) \right)d\gz}{\left(\int_{-1}^1f(\gz)^{-5}(1+x\gz)d\gz\right)^{\frac{4}{5}}}\\
\\
&=& \frac{\int_{-1}^{1} f(\gz)^{-4} \left(2a(1+x\gz)+2sx-F''(\gz) +20\left(\frac{f'(\gz)}{f(\gz)}\right)^2F(\gz)\right) d\gz}{\left(\int_{-1}^1f(\gz)^{-5}(1+x\gz)d\gz\right)^{\frac{4}{5}}},
\end{array}
\end{equation}
where the last equality comes from integration by parts and the endpoint conditions on $F(\gz)$.

Assume now that $s=-3$ and 
$a=\frac{3 \left(x^4+7\right)}{\left(1-x^2\right) \left(3-x^2\right)}$ as above and choose e.g. $F(\gz)$ as in \eqref{9dimexquasiregcscS}.
Restricting $EH(a,k,x,F,f)$ to the Sasaki cone, we assume that $f=1+c\gz$ and consider the resulting function $EH(c)$ for $-1<c<1$.
We get
$$
EH(c)=\frac{2\cdot 2^{1/5} \cdot 3^{4/5}g(c)}{(1-c^2)^{\frac{4}{5}}(3 + 3 c^2 - 5 c x - c^3 x)^{\frac{4}{5}}(1-x^2)(3-x^2)},
$$
where
$$
\begin{array}{ccl}
g(c)& = & 24 - 9 x - 4 x^2 + 12 x^3 + 4 x^4 - 3 x^5 \\
\\
&- & 8 x (5 - 2 x^2 + x^4)c\\
\\
&+ & 2 (2 + 3 x - 12 x^2 - 4 x^3 + 2 x^4 + x^5)c^2\\
\\
&+ & 16 x (1 + x^2)c^3\\
\\
&-& (4 - 3 x + 4 x^2 + 4 x^3 - x^5)c^4. 
\end{array}
$$
[As expected $EH'(c)=0$ is equivalent to \eqref{cscS9dimex} for $s=-3$ and 
$a=\frac{3 \left(x^4+7\right)}{\left(1-x^2\right) \left(3-x^2\right)}$.]
For $x=8/10$, we have that
$$EH(c)=\frac{8 \sqrt[5]{\frac{2}{15}} \left(16437-4594 c^4+16400 c^3-6533 c^2-20648 c\right)}{177 (1-c^2)^{4/5} \left(15-4 c^3+15 c^2-20 c\right)^{4/5}}$$
and the graph of this function is depicted in Figure~\ref{9dimexfigure}. It is a bit subtle to see, but it is straightforward (numerical) calculus to see that this graph has
three relative extrema; a local minimum of about $58.1955$ at $c\approx -0.4834$, a local maximum of about $58.2775$ at $c\approx -0.3123$, and a global minimum of
$\frac{1528 \sqrt[5]{\frac{182}{5}}}{59\cdot 3^{2/5}}\approx 34.2486$ at $c=4/5$. From the discussion above we know that all three critical points correspond to genuine cscS rays in the Sasaki cone.
As such, this example shows that not all cscS metrics are $CR$ $\bbt$-Yamabe minimizers and even cscS metrics that are local minima in the Sasaki-Reeb cone, might not be overall minimizers.
\end{example}

\begin{figure}[t]
\includegraphics[scale=0.5]{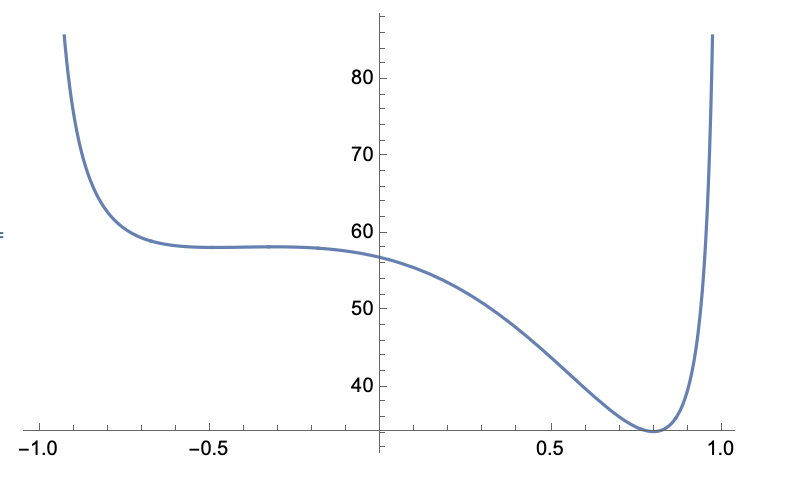}
\caption{Graph of $EH(c)$ when $x=8/10$.}
\label{9dimexfigure}
\end{figure}



\begin{thebibliography}{90}

\bibitem[Abr98]{Abreu_toric}
Miguel Abreu.
\newblock K\"{a}hler geometry of toric varieties and extremal metrics.
\newblock {\em Internat. J. Math.}, 9(6):641--651, 1998.

\bibitem[Abr10]{AbreuSasaki}
Miguel Abreu.
\newblock K\"ahler-{S}asaki geometry of toric symplectic cones in action-angle
  coordinates.
\newblock {\em Port. Math.}, 67(2):121--153, 2010.

\bibitem[ACGTF04]{HamFormsII}
Vestislav Apostolov, David M.~J. Calderbank, Paul Gauduchon, and Christina~W.
  T{\o}nnesen-Friedman.
\newblock Hamiltonian 2-forms in {K}\"ahler geometry. {II}. {G}lobal
  classification.
\newblock {\em J. Differential Geom.}, 68(2):277--345, 2004.

\bibitem[ACGTF08]{HamFormsIII}
Vestislav Apostolov, David M.~J. Calderbank, Paul Gauduchon, and Christina~W.
  T{\o}nnesen-Friedman.
\newblock Hamiltonian 2-forms in {K}{\"a}hler geometry {III}, extremal metrics
  and stability.
\newblock {\em Invent. Math.}, 173(3):547--601, 2008.

\bibitem[ACL21]{ApostolovCalderbankLegendre_weighted}
Vestislav Apostolov, David M.~J. Calderbank, and Eveline Legendre.
\newblock Weighted {K}-stability of polarized varieties and extremality of
  {S}asaki manifolds.
\newblock {\em Adv. Math.}, 391:Paper No. 107969, 63, 2021.

\bibitem[ACMY24]{AfeltraMalchiodi_EH}
Claudio Afeltra, Jih-Hsin Cheng, Andrea Malchiodi, and Paul Yang.
\newblock On the variation of the {E}instein-{H}ilbert action in
  pseudohermitian geometry.
\newblock {\em J. Reine Angew. Math.}, 813:81--102, 2024.
\newblock With an appendix by Xiaodong Wang.

\bibitem[AHP26]{Afeltra_equiv}
Claudio Afeltra, Pak~Tung Ho, and Andrea Pinamonti.
\newblock Compactness in dimension five and equivariant noncompactness for the
  {CR} {Y}amabe problem.
\newblock \href{https://arxiv.org/abs/2603.12157v1}{ arXiv:2603.12157
  [math.AP]}, 2026.

\bibitem[AM19]{ApostolovMaschler_EinsteinMaxwell}
Vestislav Apostolov and Gideon Maschler.
\newblock Conformally {K}\"ahler, {E}instein-{M}axwell geometry.
\newblock {\em J. Eur. Math. Soc. (JEMS)}, 21(5):1319--1360, 2019.

\bibitem[AMP17]{S1_Yamabe_equiv}
Bernd Ammann, Farid Madani, and Mihaela Pilca.
\newblock The {$S^1$}-equivariant {Y}amabe invariant of 3-manifolds.
\newblock {\em Int. Math. Res. Not. IMRN}, (20):6310--6328, 2017.

\bibitem[Apo19]{Apostolov_toric}
Vestislav Apostolov.
\newblock The {K}{\"a}hler geometry of toric manifolds.
\newblock Lecture Notes of CIRM winter school.
  \href{https://arxiv.org/pdf/2208.12493} {arxiv:2208.12493 [math.DG]}, 2019.

\bibitem[Aub75]{Aubin_yamabe-est}
Thierry Aubin.
\newblock \'etude d'un certain type d'\'equations diff\'erentielles non
  lin\'eaires.
\newblock {\em C. R. Acad. Sci. Paris S\'er. A-B}, 280(7):455--457, 1975.

\bibitem[Ber83]{Bergery_equiv}
Bérard Bergery.
\newblock Scalar curvature and isometry group.
\newblock {\em Kaigai Publications}, pages 9 -- 28, 1983.

\bibitem[BG00]{BG:contactNOTE}
Charles~P. Boyer and Krzysztof Galicki.
\newblock A note on toric contact geometry.
\newblock {\em J. Geom. Phys.}, 35(4):288--298, 2000.

\bibitem[BG08]{BoyerGalicki}
Charles~P. Boyer and Krzysztof Galicki.
\newblock {\em Sasakian geometry}.
\newblock Oxford Mathematical Monographs. Oxford University Press, Oxford,
  2008.

\bibitem[BHL18]{BoyerHuangLegendre_DH}
Charles Boyer, Hongnian Huang, and Eveline Legendre.
\newblock An application of the {D}uistermaat-{H}eckman theorem and its
  extensions in {S}asaki geometry.
\newblock {\em Geom. Topol.}, 22(7):4205--4234, 2018.

\bibitem[BHLTF17]{EinsteinHilbert-SasakiFutaki}
Charles~P. Boyer, Hongnian Huang, Eveline Legendre, and Christina~W. T{\o}nnesen
  Friedman.
\newblock The {E}instein-{H}ilbert functional and the {S}asaki-{F}utaki
  invariant.
\newblock {\em Int. Math. Res. Not.}, (7):1942--1974, 2017.

\bibitem[BHLTF20]{BHLTF20}
Charles~P. Boyer, Hongnian Huang, Eveline Legendre, and Christina~W.
  Tønnesen-Friedman.
\newblock {\em Some Open Problems in Sasaki Geometry}, page 143–168.
\newblock London Mathematical Society Lecture Note Series. Cambridge University
  Press, 2020.

\bibitem[BHLTF25]{BHLT25}
Charles~P. Boyer, Hongnian Huang, Eveline Legendre, and Christina~W. T{\o}nnesen
  Friedman.
\newblock Twins in {K}\"ahler and {S}asaki geometry.
\newblock {\em J. Geom. Phys.}, 216:Paper No. 105591, 29, 2025.

\bibitem[BM93]{BanyagaMolino1}
A.~Banyaga and P.~Molino.
\newblock G\'eom\'etrie des formes de contact compl\`etement int\'egrables de
  type toriques.
\newblock In {\em S\'eminaire {G}aston {D}arboux de {G}\'eom\'etrie et
  {T}opologie {D}iff\'erentielle, 1991--1992 ({M}ontpellier)}, pages 1--25.
  Univ. Montpellier II, Montpellier, 1993.

\bibitem[BTF16]{BoTo16}
Charles~P. Boyer and Christina~W. T{\o}nnesen-Friedman.
\newblock The {S}asaki join, {H}amiltonian 2-forms, and constant scalar
  curvature.
\newblock {\em J. Geom. Anal.}, 26(2):1023--1060, 2016.

\bibitem[BTF22]{BoTo22}
Charles~P. Boyer and Christina~W. T{\o}nnesen-Friedman.
\newblock The {$S^3_w$} {S}asaki join construction.
\newblock {\em J. Math. Soc. Japan}, 74(4):1335--1371, 2022.

\bibitem[BTF26]{BoTo26}
Charles~P. Boyer and Christina~W. T{\o}nnesen-Friedman.
\newblock Robustness of {CSC} {S}asaki {E}xistence under the {J}oin operation.
\newblock \href{https://doi.org/10.48550/arXiv.2604.09893}{arXiv.2604.09893
  [math.DG]}, 2026.

\bibitem[Cal82]{Calabi_extremal_metrics}
Eugenio Calabi.
\newblock Extremal {K}{\"a}hler metrics.
\newblock In {\em Seminar on Differential Geometry}, volume 102 of {\em Annals
  of Mathematical Studies}, pages 259--290. Princeton University Press, 1982.

\bibitem[Del88]{Delzant_polytope}
Thomas Delzant.
\newblock Hamiltoniens p\'{e}riodiques et images convexes de l'application
  moment.
\newblock {\em Bull. Soc. Math. France}, 116(3):315--339, 1988.

\bibitem[Don02]{Donaldson_stability_toric}
Simon~K. Donaldson.
\newblock Scalar curvature and stability of toric varieties.
\newblock {\em J. Differential Geom.}, 62(2):289--349, 2002.

\bibitem[FO18]{futakiono18}
Akito Futaki and Hajime Ono.
\newblock Volume minimization and obstructions to solving some problems in
  {K}\"ahler geometry.
\newblock {\em ICCM Not.}, 6(2):51--60, 2018.

\bibitem[FOW09]{FutakiOnoWang}
Akito Futaki, Hajime Ono, and Guofang Wang.
\newblock Transverse {K}\"ahler geometry of {S}asaki manifolds and toric
  {S}asaki-{E}instein manifolds.
\newblock {\em J. Differential Geom.}, 83(3):585--635, 2009.

\bibitem[Gam01]{Gamara}
Najoua Gamara.
\newblock The {CR} {Y}amabe conjecture---the case {$n=1$}.
\newblock {\em J. Eur. Math. Soc. (JEMS)}, 3(2):105--137, 2001.

\bibitem[Gui94]{Guillemin_toric}
Victor Guillemin.
\newblock K{\"a}hler structures on toric varieties.
\newblock {\em J. Differential Geom.}, 40(2):285--309, 1994.

\bibitem[GY01]{Gamara_Yacoub}
Najoua Gamara and Ridha Yacoub.
\newblock C{R} {Y}amabe conjecture---the conformally flat case.
\newblock {\em Pacific J. Math.}, 201(1):121--175, 2001.

\bibitem[Ho25]{Ho_CRYam-equiv}
Pak~Tung Ho.
\newblock Equivariant {CR} {Y}amabe problem.
\newblock {\em Ann. Mat. Pura Appl. (4)}, 204(1):289--306, 2025.

\bibitem[HS02]{Hwang_Singer}
Andrew~D. Hwang and Michael~A. Singer.
\newblock A momentum construction for circle-invariant {K}\"ahler metrics.
\newblock {\em Trans. Amer. Math. Soc.}, 354(6):2285--2325, 2002.

\bibitem[HV93]{HebeyVaugon_equiv}
Emmanuel Hebey and Michel Vaugon.
\newblock Le probl\`eme de {Y}amabe \'equivariant.
\newblock {\em Bull. Sci. Math.}, 117(2):241--286, 1993.

\bibitem[JL87]{JerisonLee_CRYamabe}
David Jerison and John~M. Lee.
\newblock The {Y}amabe problem on {CR} manifolds.
\newblock {\em J. Differential Geom.}, 25(2):167--197, 1987.

\bibitem[JL88]{JerisonLee_extremalsCRYam}
David Jerison and John~M. Lee.
\newblock Extremals for the {S}obolev inequality on the {H}eisenberg group and
  the {CR} {Y}amabe problem.
\newblock {\em J. Amer. Math. Soc.}, 1(1):1--13, 1988.

\bibitem[JL89]{JerisonLee_normalcoords}
David Jerison and John~M. Lee.
\newblock Intrinsic {CR} normal coordinates and the {CR} {Y}amabe problem.
\newblock {\em J. Differential Geom.}, 29(2):303--343, 1989.

\bibitem[Lah19]{Lahdili_weighted}
Abdellah Lahdili.
\newblock K\"ahler metrics with constant weighted scalar curvature and weighted
  {K}-stability.
\newblock {\em Proc. Lond. Math. Soc. (3)}, 119(4):1065--1114, 2019.

\bibitem[LeB15]{LeBrun_EinstMax}
Claude LeBrun.
\newblock The {E}instein-{M}axwell equations, {K}\"ahler metrics, and
  {H}ermitian geometry.
\newblock {\em J. Geom. Phys.}, 91:163--171, 2015.

\bibitem[LeB23]{lebrun23}
Claude LeBrun.
\newblock Yamabe invariants, homogeneous spaces, and rational complex surfaces.
\newblock {\em SIGMA Symmetry Integrability Geom. Methods Appl.}, 19:Paper No.
  027, 11, 2023.

\bibitem[Leg11]{Legendre_toricSasaki}
Eveline Legendre.
\newblock Existence and non-uniqueness of constant scalar curvature toric
  {S}asaki metrics.
\newblock {\em Compos. Math.}, 147(5):1613--1634, 2011.

\bibitem[Ler03]{Lerman:contactToric}
Eugene Lerman.
\newblock Contact toric manifolds.
\newblock {\em J. Symplectic Geom.}, 1(4):785--828, 2003.

\bibitem[LLS23]{LahdiliLegendreScarpa_EHDF}
Abdellah Lahdili, Eveline Legendre, and Carlo Scarpa.
\newblock The {E}instein-{H}ilbert functional and the {D}onaldson-{F}utaki
  {I}nvariant.
\newblock \href{https://arxiv.org/abs/2310.11625}{arXiv:2310.11625 [math.DG]},
  2023.

\bibitem[LLS25]{LahdiliLegendreScarpa_CRYam}
Abdellah Lahdili, Eveline Legendre, and Carlo Scarpa.
\newblock The {CR} {Y}amabe invariant and constant scalar curvature {S}asaki
  metrics.
\newblock \href{https://arxiv.org/abs/2509.00743}{arXiv:2509.00743 [math.DG]},
  2025.

\bibitem[LT97]{LermanTolman}
Eugene Lerman and Susan Tolman.
\newblock Hamiltonian torus actions on symplectic orbifolds and toric
  varieties.
\newblock {\em Trans. Amer. Math. Soc.}, 349(10):4201--4230, 1997.

\bibitem[LW13]{LiWang_ObataCR}
Song-Ying Li and Xiaodong Wang.
\newblock An {O}bata-type theorem in {CR} geometry.
\newblock {\em J. Differential Geom.}, 95(3):483--502, 2013.

\bibitem[Mad10]{Madani_equivconj1}
Farid Madani.
\newblock Equivariant {Y}amabe problem and {H}ebey-{V}augon conjecture.
\newblock {\em J. Funct. Anal.}, 258(1):241--254, 2010.

\bibitem[Mad12]{Madani_equivconj2}
Farid Madani.
\newblock Hebey-{V}augon conjecture {II}.
\newblock {\em C. R. Math. Acad. Sci. Paris}, 350(17-18):849--852, 2012.

\bibitem[MSY06]{MSYtoric}
Dario Martelli, James Sparks, and Shing-Tung Yau.
\newblock The geometric dual of {$a$}-maximisation for toric
  {S}asaki-{E}instein manifolds.
\newblock {\em Comm. Math. Phys.}, 268(1):39--65, 2006.

\bibitem[Rud66]{Rudin}
Walter Rudin.
\newblock {\em Real and complex analysis}.
\newblock McGraw-Hill Book Co., New York-Toronto-London, 1966.

\bibitem[ST25]{SungTakeuchi_CRYam}
Chanyoung Sung and Yuya Takeuchi.
\newblock The {CR} {Y}amabe constant and inequivalent {CR} structures.
\newblock {\em Pacific J. Math.}, 335(1):163--181, 2025.

\bibitem[Zha09]{Zhang_Kcontact}
YongBing Zhang.
\newblock The contact {Y}amabe flow on {$K$}-contact manifolds.
\newblock {\em Sci. China Ser. A}, 52(8):1723--1732, 2009.

\end{thebibliography}

\end{document}